\documentclass[10pt]{article}
 
\usepackage{hyperref}
\usepackage{amssymb}
\usepackage{amsmath}

\usepackage{color}

\usepackage{amsmath}
\usepackage{latexsym}
\usepackage{amssymb}

\usepackage{amsthm}

\font\tengoth=eufm10 at 10pt
\font\sevengoth=eufm7 at 6pt
\newfam\gothfam
\textfont\gothfam=\tengoth
\scriptfont\gothfam=\sevengoth

\newcommand{\mlabel}[1]{\marginpar{#1}\label{#1}}

\newcommand{\fS}{{\mathfrak S}}

\renewcommand{\:}{\colon}
\newcommand{\1}{\mathbf{1}}

\newcommand{\cD}{\mathcal{D}}
\newcommand{\cE}{\mathcal{E}}
\newcommand{\cF}{\mathcal{F}}

\newcommand{\cH}{\mathcal{H}}
\newcommand{\cK}{\mathcal{K}}
\newcommand{\cL}{\mathcal{L}}
\newcommand{\cM}{\mathcal{M}}

\newcommand{\subeq}{\subseteq}
\newcommand{\supeq}{\supseteq}

\newcommand{\into}{\hookrightarrow}
\newcommand{\eps}{\varepsilon}

\newcommand{\N}{{\mathbb N}}
\newcommand{\Z}{{\mathbb Z}}
\newcommand{\R}{{\mathbb R}}
\newcommand{\C}{{\mathbb C}}

\newcommand{\T}{{\mathbb T}}

\newcommand{\bS}{{\mathbb S}}

\newcommand{\bT}{{\mathbb T}}

\renewcommand{\hat}{\widehat}

\newcommand{\GL}{\mathop{{\rm GL}}\nolimits}

\newcommand{\U}{\mathop{\rm U{}}\nolimits}

\renewcommand{\Im}{\mathop{{\rm Im}}\nolimits}

\newcommand{\tr}{\mathop{{\rm tr}}\nolimits}

\newcommand{\dist}{\mathop{{\rm dist}}\nolimits}

\newcommand{\Aut}{\mathop{{\rm Aut}}\nolimits}

\newcommand{\id}{\mathop{{\rm id}}\nolimits}

\newcommand{\Spann}{\mathop{{\rm span}}\nolimits}
\newcommand{\ev}{\mathop{{\rm ev}}\nolimits}

\newcommand{\Rarrow}{\Rightarrow}
\newcommand{\nin}{\noindent} 
\newcommand{\oline}{\overline}

\newcommand{\la}{\langle}
\newcommand{\ra}{\rangle}

\newcommand{\res}{\vert}

\newcommand{\ssssarr}{\hbox to 15pt{\rightarrowfill}}
\newcommand{\sssarr}{\hbox to 20pt{\rightarrowfill}}
\newcommand{\ssarr}{\hbox to 30pt{\rightarrowfill}}
\newcommand{\sarr}{\hbox to 40pt{\rightarrowfill}}
\newcommand{\arr}{\hbox to 60pt{\rightarrowfill}}
\newcommand{\larr}{\hbox to 60pt{\leftarrowfill}}
\newcommand{\Arr}{\hbox to 80pt{\rightarrowfill}}

\def\theoremname{Theorem}
\def\propositionname{Proposition}
\def\corollaryname{Corollary}
\def\lemmaname{Lemma}
\def\remarkname{Remark}
\def\conjecturename{Conjecture} 

\def\definitionname{Definition}
\def\exercisename{Exercise}
\def\examplename{Example}
\def\examplesname{Examples}
\def\problemname{Problem}
\def\problemsname{Problems}

\def\satzname{Satz} 
\def\koroname{Korollar}
\def\folgname{Folgerung}
\def\bemerkname{Bemerkung}
\def\aufgname{Aufgabe}

\def\beisname{Beispiel}
\def\beissname{Beispiele}
\def\bewname{Beweis}

\def\@thmcounter#1{\noexpand\arabic{#1}}
\def\@thmcountersep{}
\def\@begintheorem#1#2{\it \trivlist \item[\hskip 
\labelsep{\bf #1\ #2.\quad}]}
\def\@opargbegintheorem#1#2#3{\it \trivlist
      \item[\hskip \labelsep{\bf #1\ #2.\quad{\rm #3}}]}
\makeatother
\newtheorem{theor}{\theoremname}[section]
\newtheorem{propo}[theor]{\propositionname}
\newtheorem{coro}[theor]{\corollaryname}
\newtheorem{lemm}[theor]{\lemmaname}

\newenvironment{thm}{\begin{theor}\it}{\end{theor}}
\newenvironment{theorem}{\begin{theor}\it}{\end{theor}}

\newenvironment{prop}{\begin{propo}\it}{\end{propo}}

\newenvironment{lem}{\begin{lemm}\it}{\end{lemm}}
\newenvironment{lemma}{\begin{lemm}\it}{\end{lemm}}

\newtheorem{rema}[theor]{\remarkname}

\newenvironment{rem}{\begin{rema}\rm}{\end{rema}}

\newtheorem{stepnow}[theor]{}

\newtheorem{defin}[theor]{\definitionname} 

\newenvironment{definition}{\begin{defin}\rm}{\end{defin}}
\newenvironment{defn}{\begin{defin}\rm}{\end{defin}}

\newtheorem{exerc}{\exercisename}[section]

\newtheorem{exa}[theor]{\examplename}

\newenvironment{ex}{\begin{exa}\rm}{\end{exa}}

\newtheorem{exas}[theor]{\examplesname}

\newenvironment{exs}{\begin{exas}\rm}{\end{exas}}

\newtheorem{conj}[theor]{\conjecturename}

\newtheorem{pro}[theor]{\problemname}

\newtheorem{prs}[theor]{\problemsname}

\newtheorem{aufg}{\aufgname}[section]

\newenvironment{prf}{\begin{proof}}{\end{proof}}
 
{\hfill\qed\end{trivlist}}

\newenvironment{beweis*}{\begin{trivlist}\item[\hskip%
\labelsep{\bf\bewname.\quad}]}%
{\end{trivlist}}

\newtheorem{satzn}[theor]{\satzname}

\newtheorem{koro}[theor]{\koroname}

\newtheorem{folg}[theor]{\folgname}

\newtheorem{bem}[theor]{\bemerkname}

\newtheorem{aufgn}[theor]{\aufgname}

\newtheorem{beis}[theor]{\beisname}

\newtheorem{beiss}[theor]{\beissname}

\newcommand{\braket}[2]{\la #1,#2 \ra}

\newcommand{\bD}{\mathbb D}

\newcommand{\Out}{\mathop{{\rm Out}}\nolimits}

\renewcommand{\phi}{\varphi}

\newcommand{\Han}{\mathop{{\rm Han}}\nolimits}

\renewcommand{\phi}{\varphi}
\renewcommand{\mlabel}{\label} 

\begin{document} 


\title{Reflection positivity and Hankel operators---\\
the multiplicity free case} 
\author{Maria Stella Adamo, Karl-Hermann Neeb, Jonas Schober}

\maketitle

\abstract{We analyze reflection positive representations in terms of positive 
Hankel operators. This is motivated by the fact that positive 
Hankel operators are described in terms of their 
Carleson measures, whereas the compatibility condition 
between representations and reflection positive Hilbert spaces 
is quite intricate. This leads us to the concept of a Hankel 
positive representation of triples $(G,S,\tau)$, where 
$G$ is a group, $\tau$ an involutive automorphism of $G$ and 
$S \subeq G$ a subsemigroup with $\tau(S) = S^{-1}$. 
For the triples $(\Z,\N,-\id_\Z)$, corresponding to reflection 
positive operators, and $(\R,\R_+,-\id_\R)$, corresponding to 
reflection positive one-parameter groups, we show that every 
Hankel positive representation can be made reflection positive 
by a slight change of the scalar product. A key method consists 
in using the measure $\mu_H$ on $\R_+$ defined by a positive Hankel 
operator $H$ on $H^2(\C_+)$ to define a Pick function whose imaginary 
part, restricted to the imaginary axis, provides an operator symbol for $H$.\\ 
Keywords: Hankel operator, reflection positive representation, 
Hardy space, Widom Theorem, Carleson measure, \\ 
MSC 2020: Primary 47B35; Secondary 47B32, 47B91.} 
\tableofcontents 

\vspace{1cm}

\section*{Introduction} 
\mlabel{sec:0}

This paper contributes to the operator theoretic 
background of \textit{reflection positivity}, a 
basic concept in constructive quantum field theory 
(\cite{GJ81, JOl98, JOl00, Ja08}) that recently 
required some interest from the perspective of the 
representation theory of Lie groups 
(see \cite{NO14, NO15} and the survey booklet \cite{NO18} which contains 
further references). 

The main novelty of this paper is that 
we analyze reflection positive representations in terms of positive 
Hankel operators. This is motivated by the fact that positive 
Hankel operators can be described nicely in terms of their 
Carleson measures, whereas the compatibility condition 
between representations and reflection positive Hilbert spaces 
is quite intricate. This leads us to the concept of a Hankel 
positive representation of a triple $(G,S,\tau)$, where 
$G$ is a group, $\tau$ an involutive automorphism of $G$ and 
$S \subeq G$ a subsemigroup with $\tau(S) = S^{-1}$. 
For the triples $(\Z,\N,-\id_\Z)$, corresponding to reflection 
positive operators, and $(\R,\R_+,-\id_\R)$, corresponding to 
reflection positive one-parameter groups, we show that every 
Hankel positive representation can be made reflection positive 
by a slight change of the scalar product. 

To introduce our abstract conceptual background, 
we define a  {\it symmetric semigroup} as a triple $(G,S,\tau)$, where $G$ is a group and  
$S \subeq G$ is a subsemigroup satisfying $\tau(S)^{-1} = S$, so that 
$s^\sharp := \tau(s)^{-1}$ defines an involution on $S$. 
A {\it representation of the pair $(G,S)$} is a triple 
$(\cE,\cE_+,U)$, where $U \:  G \to \U(\cE)$ is a unitary representation 
and $\cE_+ \subeq \cE$ is a closed subspace satisfying $U(S)\cE_+ \subeq \cE_+$. 
It is said to be {\it regular} if $\cE_+$ contains no 
non-zero $U(G)$-invariant subspace and the smallest $U(G)$-invariant 
subspace containing $\cE_+$ is $\cE$. 

Additional positivity is introduced by the concept of a 
{\it reflection positive Hilbert space}, which is a triple $(\cE,\cE_+,\theta)$,  
consisting of a Hilbert space 
$\cE$ with a unitary involution $\theta$ and a closed subspace  
$\cE_+$ satisfying 
\begin{equation}
  \label{eq:thetapos}
\la \xi,\xi\ra_\theta := \la \xi, \theta \xi \ra \geq 0 
\quad \mbox{  for } \quad \xi \in \cE_+.
\end{equation}
A {\it reflection positive representation} of $(G,S,\tau)$ 
is a quadruple $(\cE,\cE_+,\theta, U)$, where 
$(\cE, \cE_+, \theta)$ is a reflection positive Hilbert space 
and $(\cE,\cE_+,U)$ is a representation of the pair $(G,S)$ 
where $U$ and $\theta$ satisfy the following compatibility condition 
\begin{equation}
  \label{eq:rp1}
\theta U(g) \theta = U(\tau(g))\quad \mbox{ for } \quad g \in G.
\end{equation}
Any reflection positive representation specifies three 
representations: 
\begin{itemize}
\item[(L1)] the unitary representation $U$ of the group $G$ on $\cE$, 
\item[(L2)] the representation $U_+$ of the semigroup 
$S$ on $\cE_+$ by isometries,  
\item[(L3)] a $*$-representation $(\hat\cE, \hat U)$ of the involutive 
semigroup $(S,\sharp)$, induced by $U_+$ 
on the Hilbert space $\hat\cE$ obtained from the positive 
semidefinite form $\la \cdot, \cdot\ra_\theta$ on $\cE_+$. 
\end{itemize}

The difficulty in classifying reflection positive representations 
lies in the complicated compatibility conditions between 
$\cE_+$, $\theta$ and $U$. 
For the groups $G = \Z$ and $\R$ that we study in this paper, 
it is rather easy, resp., classical, to understand 
the regular representation $(\cE,\cE_+,U)$ of the pair $(G,S)$. 
For $(G,S) = (\Z,\N)$, this amounts to describe for a 
unitary operator $U$ all invariant subspaces $\cE_+$, and for 
$(G,S) = (\R,\R_+)$, one has to describe for a 
unitary one-parameter group $(U_t)_{t \in \R}$ 
all subspaces $\cE_+$ invariant under $(U_t)_{t > 0}$. 
Beuerling's Theorems for the disc and the upper half plane solve this problem 
in terms of inner functions (cf.~\cite[Thm.~6.4]{Pa88}, \cite{Sh64}). 
Adding to such triples $(\cE,\cE_+,U)$ 
a unitary involution $\theta$ such that $(\cE,\cE_+,\theta, U)$ is 
reflection positive is tricky because the $\theta$-positivity of $\cE_+$ 
is hard to control. 

Similarly, the description of all triples $(\cE,\theta, U)$ 
satisfying \eqref{eq:rp1} 
is the unitary representation theory of the semidirect product 
$G \rtimes \{ \id_G, \tau\}$, which is well-known for $\Z$ and $\R$. 
To fit in subspaces $\cE_+$ becomes complicated by 
the two requirements of $\theta$-positivity and $U(S)$-invariance of~$\cE_+$. 
  
The new strategy that we follow in this paper is to 
focus on the intermediate level (L2) of the representation 
$U_+$ of the involutive semigroup $(S,\sharp)$ by isometries on $\cE_+$. 
On this level, we introduce the concept 
of a {\it $U_+$-Hankel operator}. These are the 
operators $H \in B(\cE_+)$ satisfying 
\begin{equation}
  \label{eq:han1}
 H U_+(s) = U_+(s^\sharp)^* H \quad \mbox{ for } \quad s \in S.
\end{equation}
Although it plays no role for the representations of the pair $(G,S)$, 
the involution $\sharp$ on $S$ is a crucial ingredient of the concept of a 
Hankel operator.

To illustrate these structures, 
let us take a closer look at the triple $(\Z,\N_0,-\id_\Z)$, i.e., 
we study {\it reflection positive unitary operators} 
$U \in \U(\cE)$ on a reflection positive Hilbert space 
$(\cE,\cE_+,\theta)$, which means that 
\begin{equation}
  \label{eq:rp2} 
U\cE_+ \subeq \cE_+ \quad \mbox{ and } \quad \theta U \theta = U^*. 
\end{equation}
Classical normal form results for the isometry 
$S := U_+(1)$ on $\cE_+$ (assuming regularity) imply that the triple 
$(\cE,\cE_+,U)$ is equivalent to 
$(L^2(\T,\cK), H^2(\bD,\cK), U)$, where 
$\cK$ is a multiplicity space, 
\[ \bD = \{ z \in \C \: |z| < 1\} \] 
is the open unit disc, $H^2(\bD,\cK)$ is the $\cK$-valued Hardy space 
on $\bD$, and 
$(U(1)f)(z) = z f(z)$, $z \in \T$, is the multiplication operator 
corresponding to the bilateral shift on $L^2(\T,\cK)$.
Our assumption of multiplicity freeness 
means that $\cK  = \C$. In this case $U_+$-Hankel operators 
are precisely classical Hankel operators, realized as operators 
on $H^2(\bD)$. The difficult part in the classification 
of reflection positive operators consists in a description of 
all unitary involutions $\theta$ turning $(\cE,\cE_+,U) 
= (L^2(\T), H^2(\bD), U)$  
into a reflection positive representation. The compatibility 
with $U$ is easy to accommodate. It means that $\theta$ is of the 
form 
\[ (\theta_h f)(z) = h(z) f(\oline z) \quad \mbox{ with }\quad 
h \:\T \to \T,\ \  h(\oline z) = \oline{h(z)} \quad \mbox{ for } 
\quad z \in \T.\] 
The hardest part is to control the positivity of the form 
$\la \cdot,\cdot \ra_\theta$ on $\cE_+ = H^2(\bD)$. 
Here the key observation is that, if $P_+$ is the orthogonal projection 
$L^2(\T) \to H^2(\bD)$, then $H_h := P_+ \theta_h P_+^*$ 
is a Hankel operator whose positivity is equivalent to 
$(L^2(\T), H^2(\bD),\theta_h)$ being  reflection positive. 
Positive Hankel operators $H$ on $H^2(\bD)$ are most nicely classified 
in terms of their Carleson  measures $\mu_H$ on the interval 
$(-1,1)$ via the relation 
\[ \la \xi, H \eta \ra_{H^2(\bD)} 
= \int_{-1}^1 \oline{\xi(x)} \eta(x)\, d\mu_H(x)
\quad \mbox{ for } \quad \xi,\eta \in H^2(\bD).\] 
Widom's Theorem (see \cite{Wi66} and Theorem~\ref{thm:widom-disc} in the 
appendix) characterizes these measures in very explicit terms.  
Our main result on positive Hankel operators on the disc asserts that 
all these measures actually arise from reflection positive 
operators on weighted $L^2$-spaces $L^2(\T,\delta\, dz)$, where 
$\delta$ is a bounded positive weight for which $\delta^{-1}$ is also bounded 
(Theorem~\ref{thm:x.2}). 
As a consequence, the corresponding weighted Hardy space 
$H^2(\bD,\delta)$ coincides with $H^2(\bD)$, endowed with a slightly 
modified scalar product. 

The results for reflection positive one-parameter groups 
concerning $(\R,\R_+,-\id_\R)$ are similar. 
Here  the Lax--Phillips Representation Theorem 
shows that a regular multiplicity free representation 
$(\cE,\cE_+,U)$ of the pair $(\R,\R_+)$ 
is equivalent to 
$(L^2(\R), H^2(\C_+), U)$, where 
$\C_+ = \{ z \in \C \: \Im z > 0\}$ 
is the upper half-plane, $H^2(\C_+)$ is the Hardy space 
on $\C_+$, and $(U(t)f)(x) = e^{itx} f(x)$, $x \in \R$, 
is the multiplication representation. 
Again, $U_+$-Hankel operators are the classical Hankel operators on $H^2(\C_+)$ 
(cf.~\cite[p.~44]{Pa88}). 
The unitary involutions compatible with $U$ in the sense of \eqref{eq:rp1} are of the form 
\[ (\theta_h f)(x) = h(x) f(-x) \quad \mbox{ with }\quad 
h \:\R \to \T,\  \ h(-x) = \oline{h(x)} \quad \mbox{ for }  \quad x \in \R.\] 
Now $H_h := P_+ \theta_h P_+^*$ is a Hankel operator on $H^2(\C_+)$ 
and $(L^2(\R), H^2(\C_+),\theta_h)$ is reflection positive 
if and only if $H_h$ is positive. 
Instead of trying to determine all functions $h$ for which this is the case, 
we focus on positive Hankel operators $H$ on $H^2(\C_+)$ because they completely 
determine the $*$-representation in (L3). 
We prove a suitable version of Widom's Theorem for the upper half plane 
(Theorem~\ref{thm:widom-hp}) that characterizes 
the Carleson measures $\mu_H$ on $\R_+$ 
which are determined by 
\[ \la f,Hg \ra_{H^2(\C_+)} = \int_{\R_+} 
\oline{f(i\lambda)} g(i\lambda)\, d\mu_H(\lambda) 
\quad \mbox{ for } \quad f,g \in H^2(\C_+).\] 
For $\C_+$ we show that 
all these measures actually arise from reflection positive 
one-parameter groups on 
weighted $L^2$-spaces $L^2(\R,w\, dx)$, where 
$w$ is a bounded positive weight for which $w^{-1}$ is also bounded. 
As a consequence, the corresponding weighted Hardy space 
$H^2(\C_+,w\, dz)$ coincides with $H^2(\C_+)$ but is endowed with 
modified scalar product. 

Our key method of proof is to observe that the measure $\mu_H$ defines a 
holomorphic function 
\[ \kappa \: \C \setminus (-\infty,0] \to \C, \quad 
\kappa\left(z\right) := 
\int_{\R_+} \frac{\lambda}{1+\lambda^2}-\frac{1}{z+\lambda} \,d\mu_H\left(\lambda\right)\]
whose imaginary part defines a bounded function 
\[ h_H(p) :=  \frac{i}{\pi} \cdot \Im(\kappa(ip)) \] 
which is an operator symbol of $H$ 
(Theorem~\ref{thm:CarlesonRepresentant}). 
As $h(\R) \subeq i \R$, adding real constants, we obtain 
operator symbols for $H$ which are invertible in $L^\infty(\R)$, 
and this is used in Subsection~\ref{subsec:4.2} to show that, 
for $(\Z,\N,-\id_\Z)$ and $(\R,\R_+,-\id_\R)$ all 
multiplicity free regular Hankel positive representations 
can be made reflection positive by modifying the scalar product 
with an inverti\-ble intertwining operator. On the level 
of representations, this means to pass from the Hardy space 
$H^2(\bD)$, resp., $H^2(\C_+)$ to the Hardy space 
corresponding to a boundary measure with a positive bounded density 
whose inverse is also bounded. 

Since the Banach algebras $H^\infty(\bD) \cong H^\infty(\C_+)$ play a 
central role in our arguments, we decided to discuss some of their key 
features in an appendix. 
In view of the Riemann Mapping 
Theorem, this can be done for an arbitrary proper simply connected 
domain $\Omega \subeq \C$, endowed with an antiholomorphic involution 
$\sigma$ that is used to define on $H^\infty(\Omega)$ the structure 
of a Banach $*$-algebra by $f^\sharp(z) := \oline{f(\sigma(z))}$. 
By Ando's Theorem, this algebra has a unique predual, so that it carries a 
canonical weak topology, which for $H^\infty(\bD)$ and $H^\infty(\C_+)$ is defined by 
integrating boundary values against $L^1$-functions on $\T$ and $\R$, respectively. In the literature, what will be called weak topology on $H^{\infty}(\Omega)$ with respect to the canonical pairing $(H^{\infty}(\Omega)_*,H^{\infty}(\Omega))$ is also known as the weak*-topology.
For this algebra we determine in particular all weakly continuous 
positive functionals and all weakly continuous characters. \\

\nin {\bf Structure of this paper:} In the short Section~\ref{sec:1} 
we introduce the concepts on an abstract level. In particular, we define 
reflection and Hankel positive representations of symmetric 
semigroups $(G,S,\tau)$. In particular, we show that 
reflection positive representations 
are in particular Hankel positive and that 
every Hankel positive representation 
defines a $*$-representation of $(S,\sharp)$ by bounded operators on 
the Hilbert space $\hat\cE$ defined by the $H$-twisted scalar product on 
$\cE_+$. We thus obtain the same three levels (L1-3) as for reflection 
positive representations. 

In Section~\ref{sec:2}  we connect our abstract setup 
with classical Hankel operators on $H^2(\bD)$. 
We study  reflection positive representations 
of the symmetric group $(\Z,\N,-\id_\Z)$, 
i.e., reflection positive operators, 
and relate the problem of their classification
to positive Hankel operators on the Hardy space $H^2(\bD)$. 
As these operators are classified by their Carleson measures on 
the interval $(-1,1)$, we recall in Appendix~\ref{app:a} 
Widom's classical theorem characterizing the Carleson measures 
on positive Hankel operators. 
In Section~\ref{sec:3} we proceed to 
reflection positive one-parameter groups. 
In this context, the upper half plane $\C_+$ 
plays the same role as the unit disc does for the discrete context 
and any regular multiplicity free representation of the pair $(\R,\R_+)$ 
is equivalent to the multiplication representation 
on $(L^2(\R), H^2(\C_+))$. We show that in this context 
the Hankel operators coincide with the classical Hankel operators 
on $H^2(\C_+)$ and translate Widom's Theorem to an analogous 
result on the upper half plane, where we realize the Carleson 
measures on the positive half-line $\R_+$ (Theorem~\ref{thm:widom-hp}). 
The key result of Subsection~\ref{subsec:4.1} 
is Theorem~\ref{thm:CarlesonRepresentant} 
asserting that $h_H$ is an operator symbol of $H$. The applications to reflection positivity are discussed in 
Subsection~\ref{subsec:4.2}, where we prove that Hankel positive representations $(\cE, \cE_+, U, H)$ of $(\Z,\N,-\text{id}_{\Z})$ and $(\R,\R_+,-\text{id}_{\R})$ respectively, are reflection positive if we change the inner product on $\cE_+$ obtained through a symbol for $H$. 
Appendix~\ref{app:b} is devoted to the Banach $*$-algebras 
$(H^\infty(\Omega),\sharp)$ and in the short Appendix~\ref{app:k} 
we collect some formulas concerning Poisson and Szeg\"o kernels. \\

\nin {\bf Notation:} 
\begin{itemize}
\item $\R_+ = (0,\infty)$, $\bD = \{ z \in \C \: |z| <1\}$, 
$\C_+ = \R + i \R_+$ (upper half plane), $\C_r = \R_+ + i \R$ (right 
half plane), 
$\bS_\beta = \{ z \in \C \: 0 < \Im z < \beta\}$ (horizontal strip).
\item For a holomorphic function $f$ on $\C_+$, we write 
$f^*$ for its non-tangential limit function on $\R$; likewise 
for functions on $\bD$ and $\bS_\beta$. 
\item We write $\omega \: \bD \to \C_+,  \omega(z) := i \frac{1 + z}{1-z}$ 
for the Cayley transform with 
$\omega^{-1}(w) = \frac{w-i}{w+i}.$ 
\item On the circle $\T = \{ e^{i\theta} \in \C \:  \theta \in \R\}$, we 
use the length measure of total volume $2\pi$. 
\item For $w\in \C$ we write $e_w(z) := e^{zw}$ for the corresponding exponential 
function on $\C$. 
\item For a function $f \:  G \to \C$, we put 
$f^\vee(g) := f(g^{-1})$. 
\item We write $E^*$ for the dual of a Banach space~$E$.
\end{itemize}

\nin {\bf Acknowledgment:} 
We are most grateful to Daniel Belti\c t\u a for 
pointing out the references \cite{Pa88} and \cite{Pe98} on Hankel operators 
and for suggesting the connection of our work on 
reflection positivity with Hankel operators. 
We thank  Christian Berg 
for a nice short argument for the implication (b) $\Rarrow$ (c) in 
Widom's Theorem (Theorem~\ref{thm:widom-disc}).

MSA wishes to thank the Department of Mathematics, University of Erlangen, and the Mathematisches Forschungsinstitut Oberwolfach (MFO) for their hospitality. This work is part of two Oberwolfach Leibniz Fellowships with projects entitled ``Beurling--Lax type theorems and their connection with standard subspaces in Algebraic QFT'' and ``Reflection positive representations and standard subspaces in algebraic QFT''. MSA is part of the Gruppo Nazionale per l'Analisi Matematica, la Probabilit\`{a} e le loro applicationi (GNAMPA) of INdAM. MSA acknowledges the University of Rome ``Tor Vergata'' funding scheme ``Beyond Borders'' CUP E84I19002200005 and the support by the Deutsche Forschungsgemeinschaft (DFG) within the Emmy Noether grant CA1850/1-1. MSA wishes to thank Yoh Tanimoto for his insightful comments and suggestions. KHN acknowledges support by DFG-grant NE 413/10-1.

\section{Hankel operators for 
reflection positive representations} 
\mlabel{sec:1}

In this section we first recall the concept of a 
reflection positive representations of symmetric semigroups in the sense of 
\cite{NO18}. In this abstract context we introduce the notion of a 
Hankel operator (Definition~\ref{def:1.4}). 
Below it will play a key role in our analysis 
of the concrete symmetric semigroups 
$(\Z,\N_0,-\id_\Z)$ and $(\R,\R_+, -\id_\R)$, 
where it specializes to the classical concept of a Hankel 
operator on $H^2(\bD)$ and $H^2(\C_+)$, respectively.

A {\it reflection positive Hilbert space} 
is a triple $(\cE,\cE_+,\theta)$,  
consisting of a Hilbert space 
$\cE$ with a unitary involution $\theta$ and a closed subspace  
$\cE_+$ satisfying 
\[ \la \xi,\xi\ra_\theta := \la \xi, \theta \xi \ra \geq 0 
\quad \mbox{  for } \quad \xi \in \cE_+.\]
This structure immediately leads to a new Hilbert space 
$\hat\cE$ that we obtain from the positive semidefinite form  
$\la \cdot,\cdot\ra_\theta$ on $\cE_+$ by completing the quotient of 
$\cE_+$ by the subspace of null vectors. 
We write $q \: \cE_+ \to \hat \cE, \xi \mapsto \hat\xi$ for the natural map. 

\begin{defn} A {\it symmetric semigroup} 
is a triple $(G,S,\tau)$, where $G$ is a group, 
$\tau$ is an involutive automorphism of $G$, and  
$S \subeq G$ is a subsemigroup invariant under 
$s \mapsto s^\sharp := \tau(s)^{-1}$. 
\end{defn} 

In the present paper we shall only be concerned with the 
two examples $(\Z,\N_0, -\id_\Z)$ and $(\R,\R_+, -\id_\R)$. 
As it creates no additional difficulties, we formulate the concepts in 
this short section on the abstract level.

\begin{defn}   \label{def:2.2.3} Let $(G,S, \tau)$ be a symmetric semigroup. \\
\nin (a) A {\it representation of the pair $(G,S)$} is a triple 
$(\cE,\cE_+,U)$, where $U \: G \to \U(\cE)$ is a unitary representation, 
$\cE_+ \subeq \cE$ is a closed subspace and $U(S)\cE_+ \subeq \cE_+$. 
We call $(\cE,\cE_+, U)$ {\it regular} if 
\begin{equation}
  \label{eq:regular}
\oline{\Spann(U(G)\cE_+)} = \cE \quad \mbox{ and }\quad 
\bigcap_{g \in G} U(g)\cE_+ = \{0\}.
\end{equation}
This means that $\cE_+$ contains no non-zero $U(G)$-invariant 
subspace and that $\cE$ is the only closed $U(G)$-invariant subspace 
containing $\cE_+$. 

\nin (b) 
A {\it reflection positive representation of $(G,S,\tau)$} is a 
quadruple $(\cE,\cE_+,\theta, U)$, where 
$(\cE, \cE_+, \theta)$ is a reflection positive Hilbert space, 
$(\cE,\cE_+, U)$ is a representation of the pair $(G,S)$ and, in addition, 
\begin{equation}
  \label{eq:rp1b}
\theta U(g) \theta = U(\tau(g)) \quad \mbox{ for } \quad g \in G
\end{equation}
 (\cite[Def.~3.3.1]{NO18}. 
\end{defn} 

\begin{defn} \mlabel{def:1.4}
If $(S,\sharp)$ is an involutive semigroup 
and $U_+ \: S \to B(\cF)$ a representation of $S$ by bounded operators 
on the Hilbert space $\cF$, then we call 
$A \in B(\cF)$ a {\it $U_+$-Hankel operator} if 
\begin{equation}
  \label{eq:hankeldef}
 A U_+(s) = U_+(s^\sharp)^* A \quad \mbox{ for } \quad s \in S.
\end{equation}
We write $\Han_{U_+}(\cF) \subeq B(\cF)$ for the subspace of 
$U_+$-Hankel operators. 

If $U_+^\vee(s) := U_+(s^\sharp)^*$ denotes the {\it dual representation} 
of $S$ on $\cF$, then \eqref{eq:hankeldef} means that 
Hankel operators are the intertwining 
operators $(\cF,U_+) \to (\cF,U_+^\vee)$. 
\end{defn}

\begin{lem}
  \mlabel{lem:1.5} Let $(G,S,\tau)$ be a symmetric semigroup, 
$(\cE,\cE_+, U)$ be a representation of the pair $(G,S)$, 
and $P_+ \: \cE \to \cE_+$ be the orthogonal projection. 
If $A \in B(\cE)$ satisfies 
\begin{equation}
  \label{eq:interrepo}
 A U(g) = U(\tau(g))  A = U(g^\sharp)^* A
\quad \mbox{ for } \quad g \in G,
\end{equation}
then 
\[ H_A := P_+ A P_+^* \in B(\cE_+) \] 
is a $U_+$-Hankel operator for the representation of $S$ in $\cE_+$ 
by $U_+(s) := U(s)\res_{\cE_+}$. 

If, in addition, $R$ is unitary in $\cE$ 
satisfying $R U(g) R^{-1} = U(\tau(g))$ for $g \in G$, then 
$A \in B(\cH)$ satisfies \eqref{eq:interrepo} if and only if 
$A = B R$ for some $B \in U(G)'$. 
\end{lem}

\begin{prf} For the first assertion, we observe that, 
for $s \in S$ and $\xi, \eta \in \cE_+$, we have 
\begin{align*}
 \la \xi, H_A U_+(s)\eta \ra 
&= \la \xi, A U(s) \eta \ra 
\ {\buildrel {\eqref{eq:interrepo}} \over =}\ \la \xi, U(\tau(s)) A \eta \ra 
= \la U(s^\sharp) \xi, A \eta \ra \\
& =  \la U_+(s^\sharp) \xi, H_A \eta \ra 
= \la  \xi, U_+(s^\sharp)^* H_A \eta \ra.
\end{align*}
The second assertion follows from the fact that 
$B := AR^{-1}$ commutes with $U(G)$. 
\end{prf}

\begin{lem} Hankel operators have the following elementary properties: 
  \begin{itemize}
  \item[\rm(a)] If 
$H \in \Han_{U_+}(\cF)$, then $H^* \in \Han_{U_+}(\cF)$. 
  \item[\rm(b)] If $H \in \Han_{U_+}(\cF)$ 
and $B$ commutes with $U_+(S)$, then 
$HB$ and $B^*H$ are Hankel operators. 
  \end{itemize}
\end{lem}

\begin{prf} (a) If $H\in \Han_{U_+}(\cF)$ and $s \in S$, then 
\[ H^* U_+(s) 
= (U_+(s)^* H)^*
= (U_+^\vee(s^\sharp) H)^*
= (H U_+(s^\sharp))^*
= U_+(s^\sharp)^* H^*.\] 

\nin (b) Let $H \in \Han_{U_+}(\cF)$ and 
suppose that $B$ commutes with $U_+(S)$. 
Then 
\[ H B U_+(s) = H U_+(s) B = U_+^\vee(s) H B \quad \mbox{ for } \quad 
s \in S \] 
implies that $HB \in \Han_{U_+}(\cF)$. Taking adjoints, 
we obtain $B^*H= (H^*B)^* \in \Han_{U_+}(\cF)$ with~(a). 
\end{prf}

\begin{defn} (Hankel positive representations) 
Let $(G,S,\tau)$ be a symmetric semigroup. 
Then a {\it Hankel positive representation} 
is a quadruple  $(U,\cE,\cE_+, H)$, 
where $(\cE,\cE_+,U)$ is a re\-presentation of the pair $(G,S)$, 
and $H \in \Han_{U_+}(\cE_+)$ is a positive Hankel 
operator for the representation $U_+(s) := U(s)\res_{\cE_+}$ 
of $S$ by isometries on $\cE_+$.
\end{defn}

\begin{ex} \mlabel{ex:1.5} 
(a)  Let $(\cE,\cE_+,U)$ be a representation of the pair $(G,S)$ 
and $\theta \: \cE \to \cE$ a unitary involution satisfying 
$\theta U(g) \theta = U(\tau(g))$ for $g\in G$ (see \eqref{eq:rp1}). 
Then Lemma~\ref{lem:1.5} implies  that 
\[ H_\theta := P_+ \theta P_+^* \in B(\cE_+) \] 
is a $U_+$-Hankel operator. It is positive if and only if 
$(\cE,\cE_+,\theta)$ is reflection positive. 

\nin (b) The identity $\1 \in B(\cF)$ is a $U_+$-Hankel operator 
if and only if the two representations $U$ and $U^\vee$ coincide, 
i.e., if $U$ is a $*$-representation of the involutive semigroup 
$(S,\sharp)$. 
 If $U_+(S)$ consists of isometries,  this is only possible 
if all operators $U_+(s)$ are unitary and $U_+(s^\sharp) = U_+(s)^{-1}$. 
In the context of (a), this leads to the case where 
$U_+(s) \cE_+ = \cE_+$ for $s \in S$.
\end{ex}

The following proposition shows that a positive Hankel 
operator $H$ immediately leads to a \break {$*$-representation} of~$S$ 
on the Hilbert space defined by $H$ via the scalar product 
$\la \xi,\eta\ra_H := \la \xi,H\eta \ra$.

In the context of reflection positive representations 
(Example~\ref{ex:1.5}), 
the passage from the representation $(\cE_+, U_+)$ of 
$S$ by isometries to the $*$-representation 
on $(\hat\cE,\hat U)$ by contractions 
is called the {\it Osterwalder--Schrader transform}, 
see \cite{NO18} for details. In this sense, the following 
Proposition~\ref{prop:1.6} generalizes 
the Osterwalder--Schrader transform. 

\begin{prop} \mlabel{prop:1.6}
Let $U_+ \: S \to B(\cF)$ be a representation 
of the involutive semigroup $(S,\sharp)$ by bounded operators on $\cF$ 
and $H \geq 0$ be a positive $U_+$-Hankel operator on $\cF$. 
Then 
\[ \la \xi,\eta\ra_H := \la \xi,H \eta \ra_{\cF} \] 
defines a 
positive semidefinite hermitian form on $\cF$. We write 
$\hat \cF$ for the associated Hilbert space and 
$q \:  \cF \to \hat \cF$ for the canonical map. 
Then there exists a uniquely determined $*$-representation 
\begin{equation}
  \label{eq:os1}
  \hat U \:  (S,\sharp) \to B(\hat\cF) 
\quad \mbox { satisfying } \quad 
 \hat U(s) \circ q = q \circ U_+(s) \quad \mbox{ for } \quad s \in S.
\end{equation}
\end{prop}

\begin{prf} 
For every $s \in S$ and $\xi,\eta \in \cF$,  we have 
\begin{equation}
  \label{eq:sharprel1}
\la \xi, U_+(s) \eta \ra_H = \la \xi, H U_+(s) \eta \ra 
= \la \xi, U_+(s^\sharp)^* H \eta \ra 
= \la U_+(s^\sharp) \xi, H \eta \ra 
= \la U_+(s^\sharp) \xi, \eta\ra_H.
\end{equation}
If $q(\eta) = 0$, i.e., $\la \eta, H\eta \ra = 0$, 
then this relation implies that $q(U_+(s)\eta) = 0$. Therefore 
$\hat U(s) q(\eta) := q(U_+(s)\eta)$ defines a linear operator 
on the dense subspace $\cD := q(\cF) \subeq \hat\cF$. 
It also follows from \eqref{eq:sharprel1} that 
$(\hat U, \cD)$ is a $*$-representation of the involutive 
semigroup $(S,\sharp)$. 

To see that the operators $\hat U(s)$ are bounded, we observe that, 
for every $n \in \N_0$ and $\eta \in \cF$, we have 
\[ \|\hat U(s^\sharp s)^n q(\eta)\|_{\hat\cF}^2 
= \la U_+(s^\sharp s)^n \eta, H U_+(s^\sharp s)^n \eta \ra 
\leq \|H\| \|U_+(s^\sharp s)\|^{2n} \|\eta\|^2.\] 
Now \cite[Lemma~II.3.8(ii)]{Ne99} implies that 
\[ \|\hat U(s)\| \leq \sqrt{\|U_+(s^\sharp s)\|}
\leq  \max(\|U_+(s)\|, \|U_+(s^\sharp)\|) \quad \mbox{ for } \quad 
s \in S.\] 
We conclude that the operators $\hat U(s)$ are contractions, hence extend 
to operators on $\hat\cF$. We thus obtain a $*$-representation 
of $(S,\sharp)$. Clearly, this representation is uniquely determined 
by the equivariance requirement \eqref{eq:os1}. 
\end{prf}

The construction in Proposition~\ref{prop:1.6} shows that 
every Hankel positive representation \break 
$(U,\cE,\cE_+,H)$ of $(G,S,\tau)$ defines a 
$*$-representation of $S$ by bounded operators on 
the Hilbert space $\hat\cE$ defined by the $H$-twisted scalar product on 
$\cE_+$. So we obtain the same three levels (L1-3) as for reflection 
positive representations. 

\begin{rem}  Let $\cE$ be a Hilbert space, 
$\cE_+ \subeq \cE$ a closed subspace, 
and $R \in \U(\cE)$ a unitary involution  with $R(\cE_+) = \cE_+^\bot$. 
Then 
\[ A^\sharp := R^{-1} A^* R \] 
defines an antilinear involution on $B(\cE)$ leaving the subalgebra 
\[ \cM := \{ A \in B(\cE) \: A \cE_+ \subeq \cE_+ \}\] 
invariant. In fact, $A \in \cM$ implies that $A^*\cE_+^\bot\subeq \cE_+^\bot$, 
so that 
\[ A^\sharp \cE_+ = R^{-1} A^* R \cE_+ 
=  R^{-1} A^* \cE_+^\bot 
\subeq R^{-1} \cE_+^\bot = \cE_+.\] 
\end{rem} 

\begin{exs} \mlabel{ex:hankel1} (a) For 
$\cE = L^2(\T) \supeq \cE_+ = H^2(\bD)$ 
and $(Rf)(z) = \oline z f(\oline z)$, we have 
$\cM = H^\infty(\bD)$ (\cite[\S 1.8.3]{Ni19}) 
and the corresponding 
involution is given by 
\begin{equation}
  \label{eq:1.12a}
 f^\sharp(z) := \oline{f(\oline z)} \quad \mbox{ for } \quad z \in \cD.
\end{equation}
For this example Hankel operators will be discussed in 
Theorem~\ref{thm:hankel-disc}. 

\nin (b) For  $\cE = L^2(\R) \supeq 
\cE_+ = H^2(\C_+)$, we have $(Rf)(x) = f(-x)$ 
with $\cM = H^\infty(\C_+)$, endowed with the  involution 
\begin{equation}
  \label{eq:1.12b}
 f^\sharp(z) := \oline{f(-\oline z)} \quad \mbox{ for }  \quad z \in \C_+.
\end{equation}
Let $H = P_+ h R P_+^*$, $h \in L^\infty(\R)$ be a Hankel 
operator on $H^2(\C_+)$ 
(cf.\ Theorem~\ref{thm:3.5}) and $g \in H^\infty(\C_+)$. Then the corresponding 
multiplication operator $m_g$ on $H^2(\C_+)$ satisfies 
\[ H m_g = P_+ h R P_+^* m_g 
= P_+ h R m_g P_+^* = P_+ h (g^*)^\vee R P_+^*,\] 
where $(g^*)^\vee(x)= g^*(-x)$ for $x \in \R$. This also is a 
Hankel operator, where $h$ has been modified by~$(g^*)^\vee$. 
We shall use this procedure in Theorem~\ref{thm:x.1} to pass from 
Hankel positive representations to reflection positive ones. 
\end{exs}

\section{Reflection positivity 
and Hankel operators} 
\mlabel{sec:2}

In this section we connect the abstract context from the 
previous section with classical Hankel operators on $H^2(\bD)$. 
We study reflection positive operators 
as reflection positive representations 
of the symmetric group $(\Z,\N,-\id_\Z)$ and 
relate the problem of their classification
 to positive Hankel operators on the Hardy space $H^2(\bD)$ on the open 
unit disc $\bD \subeq \C$. 

\begin{defn} A {\it reflection positive operator} 
on a reflection positive Hilbert space 
$(\cE,\cE_+,\theta)$ is a unitary operator 
$U \in \U(\cE)$ such that 
\begin{equation}
  \label{eq:repo-op}
U\cE_+ \subeq \cE_+ \quad \mbox{ and } \quad \theta U \theta = U^*. 
\end{equation}
\end{defn}
It is easy to see that reflection positive operators 
are in one-to-one correspondence with  reflection positive 
representations of $(\Z,\N,-\id_\Z)$: 
If $(\cE,\cE_+,\theta,U)$ is a reflection positive 
representation of the symmetric semigroup $(\Z,\N, -\id_\Z)$, 
then  $U(1)$ is a reflection positive operator. 
If, conversely, $U$ is a reflection positive operator, 
then $U(n) := U^n$ defines a reflection positive representation 
of $(\Z,\N,-\id_\Z)$. 
Accordingly, we say that a reflection positive operator 
$U$ is {\it regular} if 
\[ \bigcap_{n \in \Z} U^n\cE_+  
= \bigcap_{n > 0} U^n\cE_+ \ {\buildrel !\over =}\ \{0\} \quad \mbox{ and  }\quad 
\oline{\bigcup_{n \in\Z} U^n\cE_+} = 
\oline{\bigcup_{n < 0} U^n\cE_+} \ {\buildrel !\over =}\  \cE\] 
(cf.\ Definition~\ref{def:2.2.3}). 
If this is the case, we obtain for 
$\cK := \cE_+ \cap (U\cE_+)^\bot$ a 
unitary equivalence from $(\cE,\cE_+, U)$ to 
\[ (\ell^2(\Z,\cK), \ell^2(\N_0,\cK), S),\] where 
$S$ is the right shift (Wold decomposition, \cite[Thm.~I.1.1]{SzNBK10}). 

We would like to classify quadruples $(\cE,\cE_+,\theta,U)$, 
where $U$ is a regular reflection positive operator, up to unitary 
equivalence. In the present paper we restrict ourselves to the 
{\it multiplicity free case}, where $\cK = \C$, so that the triple 
$(\cE,\cE_+,U)$ is equivalent to 
$(\ell^2(\Z), \ell^2(\N_0), S)$, where $S$ is the right shift. 
For our purposes, it is most convenient to 
identify $\ell^2(\Z)$ with $L^2(\T)$ and 
$\ell^2(\N_0)$ with the Hardy space $H^2(\bD)$ of the unit disc 
$\bD$, so that 
the shift operator acts by  $(Sf)(z) = z f(z)$ for $z \in \T$.

We now want to understand the possibilities for adding 
a unitary involution $\theta$ for which $H^2(\bD)$ is 
$\theta$-positive. 
For $f \: \T \to \C$ we define 
\[ f^\sharp(z) := \oline{f(\oline z)} \quad \mbox{ for } \quad z \in \T\] 
(cf.\ \eqref{eq:1.12a} and \eqref{eq:hinftysharp} in Appendix~\ref{app:b}). 
Then any  involution $\theta$ satisfying $\theta S \theta = S^{-1}$ 
has the form 
\[ \theta_h(f)(z) = h(z) f(\oline z)\quad \mbox{ for } \quad z \in \T, \] 
where $h  \in L^\infty(\T)$ satisfies $h^\sharp = h$ and 
$h(\T) \subeq \T$ (cf.~Lemma~\ref{lem:1.5}).
As any $h \in L^\infty(\T)$ defines a Hankel operator 
\begin{equation}
  \label{eq:defHh}
 H_h := P_+ \theta_h P_+^* \in B(H^2(\bD)),
\end{equation}
this leads us naturally to Hankel operators on $H^2(\bD)$. 
If $h$ is unimodular with $h^\sharp = h$, so that $\theta_h$
 is a unitary involution, then $H_h$ is positive if and only 
if $\cE_+$ is $\theta_h$-positive. 

The following theorem characterizes Hankel 
operators from several perspectives. 
Condition (a) provides the consistency 
with the abstract concept of a $U$-Hankel 
operator from Definition~\ref{def:1.4}. 
The equivalence of (a) and (c) is well known (\cite[p.~180]{Ni02}).

\begin{thm} {\rm(Characterization of Hankel Operators on the disc)} 
\mlabel{thm:hankel-disc}
Consider a bounded operator $D$ on $H^2(\bD)$, 
the shift operator $(SF)(z) =z F(z)$, 
and the multiplication operators $m_g$ defined by $g \in H^\infty(\bD)$ on 
$H^2(\bD)$. Then the following are equivalent: 
\begin{itemize}
\item[\rm(a)] The Rosenblum relation
  \begin{footnote}
    {See \cite{Ro66}, \cite[p.~205]{Ni02}.}
  \end{footnote}
$D S = S^* D$ holds for the shift operator $(Sf)(z) = z f(z)$, 
i.e., $D$ is a Hankel operator for the representation 
of $(\N,+)$ on $H^2(\bD)$ defined by $U_+(n) := S^n$. 
\item[\rm(b)] $D m_g = m_{g^\sharp}^* D$ for all $g \in H^\infty(\bD)$, 
i.e., $D$ is a Hankel operator for the representation 
of the involutive algebra $(H^\infty(\bD), \sharp)$ on 
$H^2(\bD)$ by multiplication operators. 
\item[\rm(c)] There exists $h \in L^\infty(\T)$ with 
$D = P_+ m_h R P_+^*$ for $R(F)(z) := \oline z F(\oline z)$, $z \in \bD$, 
i.e., $D$ is a bounded Hankel operator on $H^2(\bD)$ in the classical sense. 
\end{itemize}
\end{thm}

\begin{prf}  (b) $\Rarrow$ (a) is trivial. 

\nin (a) $\Rarrow$ (b): 
We recall from Example~\ref{ex:4.1}(b) that the weak 
topology on the Banach algebra $H^\infty(\bD)$ 
is defined by the linear functionals 
\begin{equation}
  \label{eq:etafdisc}
 \eta_f(h) 
= \int_\T f(z)h^*(z)\, dz
\quad \mbox{ for } \quad 
f \in L^1(\T), h \in H^\infty(\bD). \end{equation}
For $f_1, f_2 \in H^2(\bD)$, 
we observe that 
\[ \la f_1, D g f_2 \ra 
= \la D^* f_1, g f_2 \ra
= \int_{\T} \oline{(D^* f_1)^*(z)} g^*(z) f_2^*(z)\, dz 
= \eta_{\oline{(D^* f_1)^*}f_2^*}(g) \] 
(cf.\ Example~\ref{ex:4.1}), 
and 
\begin{align*}
\la g^\sharp f_1, D f_2 \ra 
&= \int_{\T} \oline{g^\sharp f_1}^*(z) (Df_2)^*(z)\, dz 
= \int_{\T} \oline{f_1^*(z)}  g^*(\oline z) (Df_2)^*(z)\, dz \\
&= \int_{\T} (f_1^\sharp)^*(z)  g^*(z) (D f_2)^*(\oline z)\, dz 
= \eta_{(f_1^\sharp)^* (D f_2)^{*,\vee}}(g),
\end{align*}
where we use the notation $h^\vee(z) = h(z^{-1})$, $z \in \T$. 
Both define weakly continuous linear functionals on $H^\infty(\bD)$ 
because $L^2(\bT) L^2(\bT) = L^1(\T)$, which by (a) coincide on polynomials. 
As these span a weakly dense subspace (Lemma~\ref{lem:polweakdense}(a)),
 we obtain equality for every $g \in H^\infty(\bD)$, which is (b). 

\nin (a) $\Leftrightarrow$ (c): It is well known that 
(a) characterizes  bounded Hankel operators on $H^2(\bD)$ 
(cf.~\cite[Thm.~2.6]{Pe98}). This relation immediately implies 
\begin{equation}
  \label{eq:sequence}
  \la z^j, D z^k \ra 
= \la z^j, D S^k 1 \ra 
= \la z^j, (S^*)^k D 1 \ra 
= \la S^k z^j, D 1 \ra 
= \la z^{j+k}, D 1 \ra,
\end{equation}
so that the matrix of $D$ is a Hankel matrix \cite[Def. 6.1.1]{Ni02}. 
The converse requires Nehari's Theorem. 
We refer to \cite[Part B, 1.4.1]{Ni02} for a nice short 
functional analytic proof. 
\end{prf}

\begin{rem}
In the proof above we have used Nehari's Theorem 
(\cite[Thm.~2.1]{Pe98},  
\cite[Part B, 1.4.1]{Ni02}, \cite[Thm.~4.7.1]{Ni19}) 
which actually contains the finer information that 
every bounded Hankel operator on $H^2(\bD)$ is of the 
form $H_h$ (see \eqref{eq:defHh}), where $h \in L^\infty(\T)$ 
can even be chosen in such a way that 
\[ \|H_h\| = \|h\|_\infty. \] 
As $H_h = 0$ if and only if $h \in H^\infty(\bD_-)$ for 
$\bD_- = \{  z \in \C \: |z| > 1 \},$ 
bounded Hankel operators on $H^2(\bD)$ 
are parametrized by the quotient space 
$L^\infty(\T)/H^\infty(\bD_-)$ (cf.\ \cite[Cor.~3.4]{Pa88}). 
As 
\begin{equation}
  \label{eq:nehari}
\|H_h\| = \dist_{L^\infty(\T)}(h, H^\infty(\bD_-)),
\end{equation}
the embedding 
$L^\infty(\T)/H^\infty(\bD_-) \into B(H^2(\bD))$ is isometric. 
\end{rem}

We now recall how positive Hankel operators can be classified by 
using Hamburger's Theorem on moment sequences. 
\begin{defn} \mlabel{def:2.5} (The Carleson measure $\mu_H$) 
Suppose that $H$ is a positive Hankel operator. 
Then \eqref{eq:sequence} shows that the sequence $(a_n)_{n \in \N_0}$ defined by 
\begin{equation}
  \label{eq:an}
 a_n := \la S^n 1, H 1 \ra 
\end{equation}
satisfies 
\[ a_{n+m} = \la S^{n+m} 1, H 1 \ra = \la S^n 1, H S^m 1\ra,\]
so that the positivity of $H$ implies that the kernel 
$(a_{n+m})_{n,m \in \N_0}$ is positive definite, i.e., 
$(a_n)_{n \in \N}$ defines a bounded positive definite function on the involutive 
semigroup $(\N_0, +,\id)$ whose bounded spectrum is $[-1,1]$. 
By Hamburger's Theorem  (\cite[Thm.~6.2.2]{BCR84}, \cite[Chap. 6]{Ni02}), 
there exists a unique positive Borel measure $\mu_H$ on $[-1,1]$ with 
\[ \int_{-1}^1 x^n\, d\mu_H(x) = a_n \quad \mbox{ for } \quad n \in \N_0.\] 
Widom's Theorem (see Theorem~\ref{thm:widom-disc} in 
Appendix~\ref{app:a}) implies that 
\[ \la f, H g \ra_{H^2(\bD)} = \int_{-1}^1 \oline{f(x)}g(x)\, d\mu_H(x) 
\quad \mbox{ for }\quad f,g \in H^2(\bD)\] 
and it characterizes the measures on $[-1,1]$ which 
arise in this context. In particular, all these measures are finite and 
satisfy $\mu_H(\{1,-1\}) = 0$. 
We call $\mu_H$ the {\it Carleson measure of $H$}. 
\end{defn}

We shall return to positive Hankel operators on the disc $\bD$ 
in Theorem~\ref{thm:x.2}.

\section{Reflection positive one-parameter groups} 
\mlabel{sec:3}

In this section we proceed from the discrete to the 
continuous by studying reflection positive one-parameter groups 
instead of single reflection positive operators. 
In this context, the upper half plane $\C_+$ 
plays the same role as the unit disc does for the discrete context.

\begin{defn} A {\it reflection positive one-parameter group} 
is a quadruple $(\cE,\cE_+, \theta, U)$ defining a reflection 
positive strongly continuous representation of the symmetric semigroup 
$(\R,\R_+, -\id_\R)$. This means that 
$(U_t)_{t \in \R}$ is a unitary one-parameter group on 
$\cE$ 
 such that 
\begin{equation}
  \label{eq:repo-onepar}
U_t\cE_+ \subeq \cE_+\quad \mbox{ for } \quad t > 0 
\quad \mbox{ and } \quad \theta U_t \theta = U_{-t}\quad \mbox{ for } 
 \quad t \in \R.
\end{equation}
As in Definition~\ref{def:2.2.3}, 
 we call a reflection positive one-parameter group 
{\it regular} if 
\[ 
\bigcap_{t \in \R} U_t\cE_+ = \bigcap_{t > 0} U_t\cE_+
\ {\buildrel!\over =}\ \{0\} \quad \mbox{ and  }\quad 
\oline{\bigcup_{t \in \R} U_t\cE_+} 
= \oline{\bigcup_{t < 0} U_t\cE_+}\ {\buildrel !\over =}\ \cE.\] 
If this is the case, then the representation 
theorem of Lax--Philipps provides a 
unitary equivalence from $(\cE,\cE_+, U)$ to 
\[ (L^2(\R,\cK), L^2(\R_+,\cK), S),\] where 
$(S_t)_{t \in \R}$ are the unitary shift operators on 
$L^2(\R,\cK)$ and $\cK$ is a Hilbert space 
(the multiplicity space) (\cite[Thm.~4.4.1]{NO18}, 
\cite{LP64, LP67, LP81}). 
\end{defn}

To classify reflection positive one-parameter groups, 
we consider in this paper the 
{\it multiplicity free case}, where $\cK = \C$. 
Again, it is more convenient to work in the spectral 
representation, i.e., to use the Fourier transform and 
to consider on 
$\cE = L^2(\R)$ the unitary multiplication operators 
\[ (S_t f)(x) = e^{itx} f(x) \quad \mbox{ for } \quad x \in \R \] 
and the Hardy space $\cE_+ := H^2(\C_+)$ which is 
invariant under the semigroup $(S_t)_{t > 0}$.

\begin{rem} \mlabel{rem:3.2} 
(Representations of $(\R,\R_+)$) 
The closed invariant subspaces 
$\cE_+ \subeq H^2(\C_+)\subeq L^2(\R)$ under the semigroup 
$(S_t)_{t > 0}$ 
are of the form $h H^2(\C_+)$ for an inner function~$h$. 
This is Beurling's Theorem for the upper half plane. It follows 
from Beurling's Theorem for the disc (\cite[Thm.~6.4]{Pa88}) 
and Lemma~\ref{lem:polweakdense} by translation 
with $\Gamma_2$ (see Theorem~\ref{thm:3.5}). 

The involutions 
$\theta$ satisfying $\theta U_t \theta = U_{-t}$ for $t \in \R$ 
are of the form 
$\theta_h = h R$, where $(Rf)(x) = f(-x)$ and $h$ is a measurable 
unimodular function on $\R$ satisfying $h^\sharp = h$, where 
$h^\sharp(x) = \oline{h(-x)}$ as in \eqref{eq:1.12b}. 
\end{rem}

\subsection{Hankel operators on $H^2(\C_+)$} 

\begin{defn} \mlabel{def:3.3} 
For $h \in L^\infty(\R)$, we define on $H^2(\C_+)$ the {\it Hankel operator} 
\[ H_h := P_+ h R P_+^*, \quad \mbox{ where } \quad (Rf)(x) := f(-x), x \in \R, \] 
$P_+ \: L^2(\R) \to H^2(\C_+)$ is the orthogonal projection, and 
$hR$ stands for the composition of $R$ with multiplication 
by $h$ (cf.~\cite[p.~44]{Pa88}).  
\end{defn}

Let 
\begin{equation*}
j_\pm : H^\infty(\C_\pm) \to L^\infty\left(\R,\C\right), \quad f \mapsto f^*
\end{equation*}
denote the isometric embedding defined by the non-tangential boundary values. 
Accordingly, we identify \(H^\infty\left(\C_\pm\right)\) with 
its image under this map in \(L^\infty\left(\R,\C\right)\).

\begin{lem} \mlabel{lem:kernelLowerHalfPlane}
For $h \in L^\infty(\R)$, the following assertions hold: 
  \begin{itemize}
  \item[\rm(a)] $H_h^* = H_{h^\sharp}$. In particular $H_h$ is hermitian 
if $h^\sharp = h$. 
  \item[\rm(b)] $H_h = 0$ if and only if $h \in H^\infty(\C_-)$.
  \item[\rm(c)] $\|H_h\| \leq \|h\|$. 
  \end{itemize}
\end{lem}

\begin{prf} (a) follows from the following relation for $f,g \in H^2(\C_+)$: 
  \begin{align*}
 \la f, H_h g \ra 
&= \int_{\R} \oline{f^*(x)} h(x) g^*(-x) \, dx 
= \int_{\R} h(-x) \oline{f^*(-x)} g^*(x) \, dx \\
&= \int_{\R} \oline{h^\sharp(x)f^*(-x)} g^*(x) \, dx 
= \la H_{h^\sharp} f, g \ra. 
  \end{align*}

\nin (b) (cf.~\cite[Cor. 4.8]{Pa88}) The operator $H_h$ vanishes if and only if 
\[ h H^2(\C_-) = \theta_h H^2(\C_+)  \subeq H^2(\C_+)^\bot = H^2(\C_-),\] which 
is equivalent to $h \in H^\infty(\C_-)$. 

\nin (c) follows from $\|P_+\| = \|R\|=1$. 
\end{prf}

The preceding lemma shows that we have a continuous linear map 
\[ L^\infty(\R)/H^\infty(\C_-) \to B(H^2(\C_+)), \quad 
[h] \mapsto H_h\] 
which is compatible with the involution $\sharp$ on the left and $*$ 
on the right. By Nehari's Theorem (\cite[Cor.~4.7]{Pa88}), this map is 
 isometric. 
As $H^\infty(\C_+) \cap H^\infty(\C_-) = \C \1$, we obtain in particular an 
embedding 
\[ H^\infty(\C_+)/\C \1 \into L^\infty(\R)/H^\infty(\C_-) \to B(H^2(\C_+)), \quad 
[h] \mapsto H_h.\] 

In Proposition~\ref{prop:1.6}, 
 we have used a positive Hankel 
operator $H$ to define a new scalar product 
that led us to a $*$-representation of $(S,\sharp)$. 
Here the key ingredient was the Hankel relation, 
an abstract form of the Rosenblum relation in Theorem~\ref{thm:hankel-disc}(b). 
As the following theorem shows, 
this relation actually characterizes Hankel operators on $H^2(\C_+)$, 
so that the classical definition (Definition~\ref{def:3.3}) and 
Definition~\ref{def:1.4} are consistent. 

\begin{thm}  \mlabel{thm:3.5}
{\rm(Characterization of Hankel Operators on the upper half plane)} 
Consider a bounded operator $C$, 
the isometries $S_t f = e_{it} f$, $t \geq 0$, 
and the multiplication operators $m_g$, $g \in H^\infty(\C_+)$ on 
$H^2(\C_+)$. We also consider the unitary isomorphism 
\begin{equation}
  \label{eq:3.5}
 \Gamma_2 \: L^2(\T) \to L^2(\R), \quad (\Gamma_2 f)(x) 
:= \frac{\sqrt{2}}{x + i} 
f\Big(\frac{x-i}{x+i}\Big)
\end{equation}
from {\rm\cite[p.~200]{Ni19}} which maps 
$H^2(\bD)$ to $H^2(\C_+)$ and the operator 
\[ D := \Gamma_2^{-1} C \Gamma_2 \: H^2(\bD) \to H^2(\bD).\] 
Then the following are equivalent: 
\begin{itemize}
\item[\rm(a)] There exists $h \in L^\infty(\R)$ with $C = H_h$, i.e., 
$C$ is a Hankel operator  in the sense of {\rm Definition~\ref{def:3.3}}. 
\item[\rm(b)] $C m_g = m_{g^\sharp}^* C$ for all 
$g \in H^\infty(\C_+)$, where $g^\sharp(z) = \oline{g(-\oline z)}$, 
 i.e., $C$ is a $U_+$-Hankel
operator for the representation 
of the involutive algebra $(H^\infty(\C_+), \sharp)$ on 
$H^2(\C_+)$ by multiplication operators $U_+(g) = m_g$. 
\item[\rm(c)] $C S_t = S_t^* C$ for all $t > 0$, i.e., 
$C$ is a $U_+$-Hankel operator for the representation 
of $\R_+$ on $H^2(\C_+)$ defined by $U_+(t)f := e_{it} f$ for $t \geq 0$. 
\item[\rm(d)] $D$ is a Hankel operator on $H^2(\bD)$.
\end{itemize}
\end{thm}

\begin{prf} (a) $\Rarrow$ (b): Suppose that 
$C = H_h$ for some $h \in L^\infty(\R)$. 
For $f_1, f_2 \in H^2(\C_+)$ we then have 
\[ \la f_1, H_h g f_2 \ra 
= \int_{\R} \oline{f_1^*(x)} h(x) g^*(-x) f_2^*(-x)\, dx 
= \int_{\R} \oline{g^{*,\sharp}(x) f_1^*(x)} h(x)  f_2^*(-x)\, dx 
= \la g^{\sharp} f_1, H_h f_2 \ra,\] 
which is (b). 

\nin (b) $\Rarrow$ (c) follows from $e_{it}^\sharp = e_{it}$ for $t > 0$. 

\nin (c) $\Rarrow$ (b): For $f_1, f_2 \in H^2(\C_+)$ and 
$g \in H^\infty(\C_+)$, 
we observe that 
\[ \la f_1, C g f_2 \ra 
= \la C^* f_1, g f_2 \ra
= \int_{\R} \oline{C^* f_1}^*(x) g^*(x) f_2^*(x)\, dx 
= \eta_{\oline{C^* f_1}f_2}(g) \] 
(see Example~\ref{ex:4.1} for the functionals $\eta_f$) 
and 
\begin{align*}
\la g^\sharp f_1, C f_2 \ra 
&= \int_{\R} \oline{(g^\sharp f_1)^*}(x) (Cf_2)^*(x)\, dx 
= \int_{\R} \oline{f_1^*(x)}  g^*(-x) (Cf_2)^*(x)\, dx \\
&= \int_{\R} (f_1^\sharp)^*(x)  g^*(x) (C f_2)^*(-x)\, dx 
= \eta_{f_1^\sharp (C f_2)^\vee}(g),
\end{align*}
where we use the notation $h^\vee(x) := h(-x)$ for $x \in \R$. 
Both define weakly continuous linear functionals on $H^\infty(\C_+)$, 
which 
by (c) coincide on the functions $e_{it}$, $t > 0$. 
As these span a weakly dense subspace (Lemma~\ref{lem:polweakdense}(b)), 
we obtain 
equality for every $g \in H^\infty(\C_+)$, which is (b). 

\nin (b) $\Leftrightarrow$ (d): The Cayley transform 
$\omega \: \bD \to \C_+,  \omega(z) := i \frac{1 + z}{1-z}$ defines 
an isometric isomorphism 
$L^\infty(\T) \to L^\infty(\R), g \mapsto g \circ \omega^{-1}$ 
which restricts to an isomorphism 
$H^\infty(\bD) \to H^\infty(\C_+)$ 
and satisfies 
\[ \Gamma_2 \circ m_g = m_{g \circ \omega^{-1}} \circ \Gamma_2.\] 
Therefore (b) is equivalent to 
\[ D m_{g \circ \omega} = m_{g^\sharp \circ \omega}^* D 
\quad \mbox{ for } \quad g \in H^\infty(\C_+),\] 
which is (d) by Theorem~\ref{thm:hankel-disc}. 

\nin (d) $\Rarrow$ (a): Suppose that $D = D_k$ as in 
Theorem~\ref{thm:hankel-disc}. 
For $f \in H^2(\bD)$, we then have for $x \in \R$ 
\begin{align*}
(C \Gamma_2(f))(x) 
&= \Gamma_2(Df)(x) 
= \frac{\sqrt{2}}{x + i} (Df)^*(\omega^{-1}(x)) 
= \frac{\sqrt{2} k(\omega^{-1}(x))}{x + i} f^*(\oline{\omega^{-1}(x)}) \\
&= k(\omega^{-1}(x))\frac{i-x}{(i+x)}
\frac{\sqrt{2}}{(-x+i)}f^*(\omega^{-1}(-x))
= k(\omega^{-1}(x))\frac{i-x}{i+x} \Gamma_2(f)^*(-x). 
\end{align*}
The assertion now follows with 
\begin{equation}
  \label{eq:khrel}
h(x) := k(\omega^{-1}(x))\frac{i-x}{i+x} = -k(\omega^{-1}(x))\omega^{-1}(x) 
\end{equation}
(cf.\ \cite[Thm.~4.6]{Pa88}). 
\end{prf}

\subsection{Widom's Theorem for the upper half-plane} 

In this subsection we translate Widom's Theorem 
(Theorem~\ref{thm:widom-disc})
characterizing the Carleson measures of positive Hankel operators 
on the disc to a corresponding result on the upper half plane. 
This is easily achieved by using Theorem~\ref{thm:3.5} for the translation 
process.

Let $H$ be a positive Hankel operator on $H^2(\C_+)$. 
For $t \geq 0$, the exponential functions 
$e_{it}(z) = e^{itz}$ in $H^\infty(\C_+)$ satisfy 
$e_{it}^\sharp = e_{it}$. Therefore the function 
\begin{equation}
  \label{eq:anb}
\phi_H \: \R_+ \to \R,\quad \phi_H(t) := \la e_{it/2}, H e_{it/2} \ra_{H^2(\C_+)} 
\end{equation} 
satisfies 
\[ \phi_H(t+s) 
= \la e_{i(t+s)/2}, H e_{i(t+s)/2} \ra_{H^2(\C_+)}  
= \la e_{it}, H e_{is} \ra_{H^2(\C_+)}  \quad \mbox{ for } \quad s,t> 0,\]
so that  the kernel $(\phi_H(t+s))_{t,s > 0}$ 
is positive definite. This means that $\phi_H$ is a 
positive defi\-nite function on the involutive 
semigroup $(\R_+, +,\id)$ bounded on $[1,\infty)$. 
By the Hausdorff--Bernstein--Widder Theorem (\cite[Thm.~6.5.12]{BCR84}, 
\cite[Thm.~VI.2.10]{Ne99}), 
there exists a unique positive Borel measure $\mu_H$ on $[0,\infty)$  with 
\begin{equation}
  \label{eq:3.2}
\phi_H(t) = \int_0^\infty e^{-\lambda t} \, d\mu_H(\lambda)
 \quad \mbox{ for } \quad t > 0.
\end{equation}
Widom's Theorem for $\C_+$ (Theorem~\ref{thm:widom-hp} below) 
now implies that 
\[ \la f, H g \ra_{H^2(\C_+)} = \int_0^\infty \oline{f(i\lambda)}g(i\lambda)
\, d\mu_H(\lambda) 
\quad \mbox{ for }\quad f,g \in H^2(\C_+)\] 
and it characterizes the measures $\mu_H$ on $[0,\infty)$ which 
correspond to positive bounded Hankel operators.
In particular, all these measures satisfy $\mu_H(\{0\}) = 0$. 

\begin{defn} The measure $\mu_H$ on $\R_+$ is called the Carleson 
measure of $H$. 
\end{defn}

\begin{thm} \mlabel{thm:widom-hp} {\rm(Widom's Theorem for 
the upper half-plane)} 
For a positive Borel measure $\mu$ on~$\R_+$, 
we consider the measure $\rho$ on $\R_+$ defined by 
\[  d\rho(\lambda) := \frac{d\mu(\lambda)}{1 + \lambda^2}.\] 
Then  the following are equivalent: 
\begin{itemize}
\item[\rm(a)] There exists an $\alpha \in \R$ with 
  \begin{equation}
    \label{eq:alphaesti}
 \int_{\R_+} |f(i\lambda)|^2\, d\mu(\lambda) \leq \alpha \|f\|^2 
\quad \mbox{ for } \quad f \in H^2(\C_+),
  \end{equation}
i.e., $\mu$ is the Carleson measure of a positive Hankel operator 
on $H^2(\C_+)$. 
\item[\rm(b)] $\rho((0,x)) = O(x)$ and $\rho((x^{-1},\infty)) = O(x)$ 
for $x \to 0+$. 
\end{itemize}
If these conditions are satisfied, then 
$\rho(\R_+) < \infty$ and there exist $\beta, \gamma > 0$ such that 
\[ \rho((0,\eps])  \leq \beta \eps \quad \mbox{and} 
\quad  \rho([t,\infty)) \leq \frac \gamma t
\quad \mbox { for every } \quad  \eps,t \in \R_+.\]
\end{thm}

\begin{prf} Condition (a) is equivalent to the existence 
of a positive Hankel operator $C$ on $H^2(\C_+)$ 
with $\mu = \mu_C$. Let $D$ be the corresponding 
Hankel operator on $H^2(\bD)$ (Theorem~\ref{thm:3.5}) 
and consider the diffeomorphism 
\[ \gamma \: \R_+ \to (-1,1), \quad 
\gamma(\lambda) = \frac{\lambda -1}{\lambda +1} = \omega^{-1}(i\lambda).\] 
For $f \in H^2(\bD)$, we then have 
\begin{align*}
& \int_{-1}^1 |f(t)|^2\, d\mu_D(t) 
= \la f, D f \ra_{H^2(\bD)} 
= \la \Gamma_2(f), C \Gamma_2(f) \ra_{H^2(\C_+)} \\
&= \int_{\R_+} |\Gamma_2(f)(i\lambda)|^2\, d\mu_C(\lambda) 
= 2 \int_{\R_+} \frac{|f(\omega^{-1}(i\lambda))|^2}
{(1+\lambda)^2}\, d\mu_C(\lambda) \\
&= 2 \int_{\R_+} \frac{|f(\gamma(\lambda))|^2} 
{(1+\lambda)^2}\, d\mu_C(\lambda) 
= 2 \int_{-1}^1 \frac{|f(t)|^2} 
{(1+\gamma^{-1}(t))^2}\, d(\gamma_*\mu_C)(t). 
\end{align*}
As $\gamma^{-1}(t) = -i \omega(t) = \frac{1+t}{1-t}$ and 
$1 + \frac{(1+t)}{(1-t)} = \frac{2}{(1-t)},$ 
it follows that 
\[ d\mu_D(t) = \frac{(1-t)^2}{2} d(\gamma_*\mu_C)(t).\]
We conclude that 
\begin{align*}
 \mu_D((1-x,1)) 
&= \int_{1-x}^1 \frac{(1-t)^2}{2} d(\gamma_*\mu_C)(t) 
= \int_{\gamma^{-1}(1-x)}^\infty \frac{(1-\gamma(\lambda))^2}{2} 
d\mu_C(\lambda)\\
&= \int_{\frac{2}{x}-1}^\infty \frac{2}{(\lambda + 1)^2} 
d\mu_C(\lambda) 
= 2 \int_{\frac{2}{x}-1}^\infty \frac{1 + \lambda^2}{(\lambda + 1)^2} 
d\rho(\lambda).
\end{align*}
Therefore $\mu_D((1-x,1))$ has for $x \to 0^+$ the same asymptotics  
as $\rho((x^{-1},\infty))$. 
Likewise 
\begin{align*}
 \mu_D((-1,-1+x)) 
 &= \int_{-1}^{-1+x} \frac{(1-t)^2}{2} d(\gamma_*\mu_C)(t) 
= \int_0^{\gamma^{-1}(x-1)} \frac{(1-\gamma(\lambda))^2}{2} 
d\mu_C(\lambda)\\
&= \int_0^{\frac{x}{2-x}} \frac{2}{(\lambda + 1)^2} 
d\mu_C(\lambda) = 
 2\int_0^{\frac{x}{2-x}} \frac{1 + \lambda^2}{(\lambda + 1)^2} 
d\rho(\lambda).
\end{align*}
This shows that $\mu_D((-1,-1+x))$ has for $x \to 0+$ the same asymptotics 
as $\rho((0,x))$. 
Therefore the assertion follows from Widom's Theorem for the disc 
(Theorem~\ref{thm:widom-disc}).

Now we assume that $\rho$ satisfies (b). 
Then there exist $\beta', \gamma' > 0$ and 
\(\eps_0,t_0 \in \R_+\) such that
\[\frac{\rho\left(\left(0,\eps\right]\right)}\eps \leq \beta' \quad \text{and} \quad \rho\left(\left[t,\infty\right)\right)t \leq \gamma'
\quad \mbox{ for every  } \quad \eps \leq \eps_0, \ 
t \geq t_0.\] 
Then 
\[ \rho(\R_+) = \rho\left(\left(0,\eps_0\right)\right) + \rho\left(\left[\eps_0,t_0\right]\right) + \rho\left(\left(t_0,\infty\right)\right)
\leq \beta' + \rho\left(\left[\eps_0,t_0\right]\right) + \gamma' 
 < \infty.\] 
For \(\eps>\eps_0\) and \(t<t_0\), we now find 
\[\frac{\rho\left(\left(0,\eps\right]\right)}\eps \leq \frac{\rho\left(\R_+\right)}{\eps_0} \quad \text{and} \quad \rho\left(\left[t,\infty\right)\right)t \leq \rho\left(\R_+\right)t_0.\]
This completes the proof. 
\end{prf}

\subsection{The symbol kernel of a positive Hankel operator} 

\begin{definition} 
Let $H$ be a Hankel operator on $H^2(\C_+)$ 
and 
\[  Q(z,w) = Q_w(z) = \frac{1}{2\pi} \frac{i}{z - \oline w}  \] 
be the Szeg\"o kernel of $\C_+$ 
(cf.\ Appendix~\ref{app:k}). Then we 
associate to $H$ its {\it symbol kernel}, i.e., the kernel 
\begin{equation}
  \label{eq:eq:KHdef}
Q_H(z,w) :=   \braket{Q_z}{HQ_w} 
=  (HQ_w)(z) =  \oline{(H^* Q_z)(w)}. 
\end{equation}
Clearly, $Q_H$ is holomorphic in the first argument and 
antiholomorphic in the second argument. 
\end{definition}

By \cite[Lemma~I.2.4]{Ne99}, 
the Hankel operator $H$ is positive if and only if its symbol kernel 
$Q_H$ is positive definite. 
Suppose that this is the case and let $\mu_H$ be the corresponding 
Carleson measure on $\R_+$. 
Then 
\begin{align} \label{eq:KHdef} 
Q_H(z,w) &=  \int_0^\infty \oline{Q_z(i\lambda)} 
Q_w(i\lambda)\, d\mu_H(\lambda)
=  \frac{1}{4\pi^2}\int_0^\infty \frac{d\mu_H(\lambda)}
{(-i\lambda - z)(i\lambda - \oline w)}\notag \\
&=  \frac{1}{4\pi^2}\int_0^\infty \frac{d\mu_H(\lambda)}{(\lambda - iz)(\lambda + i \oline w)}.\end{align}

\begin{defn}
From Widom's Theorem for the upper half plane 
(Theorem~\ref{thm:widom-hp}), we know that the measure 
$\frac{d\mu(\lambda)}{1 + \lambda^2}$ is finite, so that, 
\begin{equation} \label{eq:kappa}
\kappa\left(z\right) := 
\int_{\R_+} \frac{\lambda}{1+\lambda^2}-\frac{1}{z+\lambda} 
\,d\mu_H\left(\lambda\right)
\end{equation}
defines a holomorphic function on $\C \setminus (-\infty,0]$ 
(\cite[Ch.~II, Thm.~1]{Do74}). 
\end{defn}

For $z,w \in \C_r$, we then have 
\[ \kappa(z) - \kappa(w) 
= \int_{\R_+} \frac{1}{w + \lambda} - \frac{1}{z + \lambda}\, d\mu_H(\lambda)
= \int_{\R_+} \frac{z-w}{(w + \lambda)(z + \lambda)}\, d\mu_H(\lambda),\] 
so that 
\begin{equation}
  \label{eq:kerl-rel} 
  \frac{\kappa(z) - \kappa(w)}{z-w} 
=  \int_{\R_+} \frac{d\mu_H(\lambda)}{(w + \lambda)(z + \lambda)}
= 4\pi^2 Q_H(iz, i \oline w). 
\end{equation}

\section{Schober's representation theorem} 
\mlabel{sec:4}

In this section we explain how to find for every positive Hankel 
operator $H$ on $H^2(\C_+)$ an explicit bounded function 
$h_H \in L^\infty(\R)$ with values in $i \R$ such  that 
$h_H^\sharp = h_H$ and $H$ is the corresponding Hankel operator, i.e., $H_{h_H}= H$. This supplements Nehari's classical theorem by a constructive component. 
Adding non-zero real constants then leads to functions 
$f$ in the unit group of $L^\infty(\R)$ with $H_f = H$, and 
we shall use this to shows that all Hankel positive one-parameter groups 
are actually reflection positive for a slightly modified scalar product. 

\subsection{An operator symbol for $H$} 
\mlabel{subsec:4.1}

\begin{theorem}\label{thm:CarlesonRepresentant}
Let \(H\) be a positive Hankel operator on \(H^2(\C_+)\) with 
Carleson measure $\mu_H$ and define
\[ h_H \: \R \to i \R, \quad h_H\left(p\right) :=
 \frac i\pi \cdot \int_{\R_+} \frac {p}{\lambda^2+p^2} \,d\mu_H\left(\lambda\right).\]
Then \(h_H \in L^\infty\left(\R,\C\right)\) and 
the associated Hankel operator $H_{h_H}$ equals~$H$. 
\end{theorem}

\begin{proof} {\bf Part 1:} 
We first show that \(h_H\) is bounded. 
Let $d\rho\left(\lambda\right) = \frac{d\mu_H\left(\lambda\right)}{1+\lambda^2}$ 
be the finite measure on \(\R_+\) from Theorem~\ref{thm:widom-hp}. 
Then we have
\[\int_{\R_+} \frac {p}{\lambda^2+p^2} \,d\mu_H\left(\lambda\right) = \int_{\R_+} \frac {p\left(1+\lambda^2\right)}{\lambda^2+p^2} \,d\rho\left(\lambda\right).\]
For the integrand 
\[ f_p\left(\lambda\right) :=  \frac {p\left(1+\lambda^2\right)}{\lambda^2+p^2}
\quad \mbox{ we have } \quad 
f_p'\left(\lambda\right) = \frac {2p\left(p^2-1\right)\lambda}{\left(\lambda^2+p^2\right)^2}.\]
Hence the function $f_p$ is increasing for $p \geq 1$, and therefore
\[\int_{\left(0,1\right]} f_p\left(\lambda\right)d\rho\left(\lambda\right) \leq f_p\left(1\right) \int_{\left(0,1\right]}d\rho\left(\lambda\right) = \frac{2p}{1+p^2} \cdot \rho\left(\left(0,1\right]\right) \leq \rho\left(\left(0,1\right]\right).\]
Now, let \(\gamma\) be the constant from 
Theorem~\ref{thm:widom-hp}. Then 
integration by parts (cf.~Lemma~\ref{lem:intpart}) leads for 
\(p \geq 1\) to 
\begin{align*}
\int_{\left(1,\infty\right)} f_p\left(\lambda\right)d\rho\left(\lambda\right) &= \rho\left(\left(1,\infty\right)\right) f_p\left(1\right) + \int_{\left(1,\infty\right)} \rho\left(\left(t,\infty\right)\right) f_p'\left(t\right)\,dt
\\&\leq \rho\left(\left(1,\infty\right)\right) \frac{2p}{1+p^2} + \int_{\left(1,\infty\right)} \frac \gamma t \cdot \frac {2p\left(p^2-1\right)t}{\left(t^2+p^2\right)^2}\,dt
\\&\leq \rho\left(\left(1,\infty\right)\right) \cdot 1 + \gamma \left(p^2-1\right) \left[\frac{\frac{tp}{t^2+p^2}+\arctan\left(\frac tp\right)}{p^2}\right]_1^\infty
\\&= \rho\left(\left(1,\infty\right)\right) + \gamma \,\frac{p^2-1}{p^2} \left[\frac \pi 2 - \frac{p}{1+p^2} - \arctan\left(\frac 1p\right)\right] \leq \rho\left(\left(1,\infty\right)\right) + \frac {\gamma \pi}2.
\end{align*}
So, for every \(p \geq 1\), we have
\begin{align*}
\int_{\R_+} \frac {p}{\lambda^2+p^2} \,d\mu_H\left(\lambda\right) &= \int_{\R_+} f_p\left(\lambda\right) \,d\rho\left(\lambda\right) = \int_{\left(0,1\right]} f_p\left(\lambda\right)d\rho\left(\lambda\right) + \int_{\left(1,\infty\right)} f_p\left(\lambda\right)d\rho\left(\lambda\right)
\\& \leq \rho\left(\left(0,1\right]\right) + \rho\left(\left(1,\infty\right)\right) + \frac {\gamma \pi}2 = \rho\left(\R_+\right) + \frac {\gamma \pi}2.
\end{align*}
For \(p \in \left(0,1\right)\), the function \(f_p\) is decreasing and therefore
\[\int_{\left(1,\infty\right)} f_p\left(\lambda\right)d\rho\left(\lambda\right) \leq f_p\left(1\right) \int_{\left(1,\infty\right)}d\rho\left(\lambda\right) = \frac{2p}{1+p^2} \cdot \rho\left(\left(1,\infty\right)\right) \leq \rho\left(\left(1,\infty\right)\right).\]
Now, let \(\beta\) be the constant from 
Theorem~\ref{thm:widom-hp}. Then, for \(p < 1\), we have
\begin{align*}
\int_{\left(0,1\right]} f_p\left(\lambda\right)d\rho\left(\lambda\right) &= \rho\left(\left(0,1\right]\right) f_p\left(1\right) - \int_{\left(0,1\right]} \rho\left(\left(0,t\right]\right) f_p'\left(t\right)\,dt
\\&\leq \rho\left(\left(0,1\right]\right) \frac{2p}{1+p^2} - \int_{\left(0,1\right]} \beta t \cdot \frac {2p\left(p^2-1\right)t}{\left(t^2+p^2\right)^2}\,dt
\\&\leq \rho\left(\left(0,1\right]\right) \cdot 1 + \beta \left(1-p^2\right) \left[\arctan\left(\frac tp\right)-\frac{tp}{t^2+p^2}\right]_0^1
\\&= \rho\left(\left(0,1\right]\right) + \beta \left(1-p^2\right) \left[\arctan\left(\frac 1p\right) - \frac{p}{1+p^2}\right] \leq \rho\left(\left(0,1\right]\right) + \frac {\beta \pi}2.
\end{align*}
So, for every \(p \in \left(0,1\right)\), we have
\begin{align*}
\int_{\R_+} \frac {p}{\lambda^2+p^2} \,d\mu_H\left(\lambda\right) &= \int_{\R_+} f_p\left(\lambda\right) \,d\rho\left(\lambda\right) = \int_{\left(0,1\right]} f_p\left(\lambda\right)d\rho\left(\lambda\right) + \int_{\left(1,\infty\right)} f_p\left(\lambda\right)d\rho\left(\lambda\right)
\\& \leq \rho\left(\left(0,1\right]\right) + \frac {\beta \pi}2 + \rho\left(\left(1,\infty\right)\right) = \rho\left(\R_+\right) + \frac {\beta \pi}2.
\end{align*}
Therefore, for every \(p \in \R_+\), we have
\[\left|h_H(p)\right| = \frac 1\pi \int_{\R_+} \frac {p}{\lambda^2+p^2} \,d\mu_H\left(\lambda\right) \leq \frac 1\pi \rho\left(\R_+\right) + \frac 1 2 \max\{\beta,\gamma\}.\]
Since \(h_H(-p) = -h_H(p)\), this yields
\[\left\lVert h_H\right\rVert_\infty \leq \frac 1\pi \rho\left(\R_+\right) + \frac 12 \max\{\beta,\gamma\}\]
and therefore \(h_H \in L^\infty\left(\R,\C\right)^\sharp\), where \(h_H^\sharp = h_H\) follows by \(h_H(-p) = -h_H(p) = \overline{h_H(p)}\).

\nin {\bf Part 2:} 
For the second statement, we recall the function 
\[\kappa: \C \setminus \left(-\infty,0\right] \to \C, \quad 
\kappa(z) =  \int_{\R_+} \frac \lambda{1+\lambda^2}-\frac 1{\lambda+z} \,d\mu_H\left(\lambda\right)\]
from \eqref{eq:kappa}. 
Then, for \(p \in \R^\times\), we have
\begin{align*}
\mathrm{Im}\left(\kappa\left(ip\right)\right) &= \mathrm{Im}\left(\int_{\R_+} \frac \lambda{1+\lambda^2}-\frac 1{\lambda+ip} \,d\mu_H\left(\lambda\right)\right)
\\&= \mathrm{Im}\left(\int_{\R_+} \frac \lambda{1+\lambda^2} - \frac {\lambda-ip}{\lambda^2+p^2} \,d\mu_H\left(\lambda\right)\right) = \int_{\R_+} \frac {p}{\lambda^2+p^2} \,d\mu_H\left(\lambda\right),
\end{align*}
so
\begin{equation}
  \label{eq:hkapparel}
h_H\left(p\right) = \frac i\pi \cdot \mathrm{Im}\left(\kappa\left(ip\right)\right).
\end{equation}
For the real part, we get
\begin{align*}
\mathrm{Re}\left(\kappa\left(ip\right)\right) &= \mathrm{Re}\left(\int_{\R_+} \frac \lambda{1+\lambda^2}-\frac 1{\lambda+ip} \,d\mu_H\left(\lambda\right)\right) = \mathrm{Re}\left(\int_{\R_+} \frac \lambda{1+\lambda^2}-\frac {\lambda-ip}{\lambda^2+p^2} \,d\mu_H\left(\lambda\right)\right)
\\&= \int_{\R_+} \frac \lambda{1+\lambda^2}- \frac {\lambda}{\lambda^2+p^2} \,d\mu_H\left(\lambda\right) = \left(p^2-1\right) \int_{\R_+} \frac {\lambda}{\left(1+\lambda^2\right)\left(\lambda^2+p^2\right)} \,d\mu_H\left(\lambda\right)
\end{align*}
and therefore
\begin{align*}
\left|\mathrm{Re}\left(\kappa\left(ip\right)\right)\right| &=  \left|p^2-1\right| \int_{\R_+} \frac {\lambda}{\left(1+\lambda^2\right)\left(\lambda^2+p^2\right)} \,d\mu_H\left(\lambda\right)
\\&\leq \left|p^2-1\right| \int_{\R_+} \frac {4\lambda}{\left(1+\lambda\right)^2\left(\left|p\right|+\lambda\right)^2} \,d\mu_H\left(\lambda\right).
\end{align*}
For \(p \in \R^\times\), we now define the function
\[n_p: \C_+ \to \C, \quad n_p(z) =  \frac{2\sqrt{z}}{\left(1-iz\right)\left(\left|p\right|-iz\right)},\]
where by \(\sqrt{\cdot}\) we denote the inverse of the function
$\C_r \cap \C_+  \to \C_+,  z \mapsto z^2.$ 
Then \(n_p\) is holomorphic on \(\C_+\) and for \(y > 0\), we have
\[\left|n_p\left(x+iy\right)\right|^2 = \frac{4\sqrt{x^2+y^2}}{\left((1+y)^2+x^2\right)\left((\left|p\right|+y)^2+x^2\right)} \leq \frac{4\sqrt{x^2+y^2}}{\left(1+y^2+x^2\right)\left(p^2+x^2\right)} \leq \frac{2}{p^2+x^2},\]
so
\[\sup_{y > 0} \int_\R \left|n_p\left(x+iy\right)\right|^2 dx \leq \int_\R \frac{2}{p^2+x^2} \,dx = \frac{2\pi}{\left|p\right|} < \infty\]
and therefore \(n_p \in H^2\left(\C_+\right)\). Since 
\(\mu_H\) is a Carleson measure, 
by Theorem~\ref{thm:widom-hp}(a), 
there is a constant \(\alpha \geq 0\) such that
\[\int_{\R_+} \oline{f\left(i\lambda\right)}g\left(i\lambda\right) \,d\mu_H\left(\lambda\right) \leq \alpha \left\lVert f\right\rVert_2\left\lVert g\right\rVert_2 \quad \mbox{ for every } \quad f,g \in H^2\left(\C_+\right).\]
 Then
\begin{align*}
\left|\mathrm{Re}\left(\kappa\left(ip\right)\right)\right| &\leq \left|p^2-1\right| \int_{\R_+} \frac {4\lambda}{\left(1+\lambda\right)^2\left(\left|p\right|+\lambda\right)^2} \,d\mu_H\left(\lambda\right) = \left|p^2-1\right| \int_{\R_+} \left|n_p\left(i\lambda\right)\right|^2 \,d\mu_H\left(\lambda\right)
\\&\leq \left|p^2-1\right| \alpha \left\lVert n_p\right\rVert_2^2 = \alpha \left|p^2-1\right| \int_\R \frac{4\left|x\right|}{\left(1+x^2\right)\left(p^2+x^2\right)} \,dx
\\&=4\alpha \left|p^2-1\right| \int_0^\infty \frac{2x}{\left(1+x^2\right)\left(p^2+x^2\right)} \,dx = 4\alpha  \left|\int_0^\infty \frac{2x}{1+x^2} - \frac{2x}{p^2+x^2} \,dx\right|
\\&=4\alpha \left|\left[\log\left(1+x^2\right)-\log\left(p^2+x^2\right)\right]_0^\infty \right| = 4\alpha \left|\left[\log\left(\frac{1+x^2}{p^2+x^2}\right)\right]_0^\infty \right| = 8 \alpha \left|\log\left(\left|p\right|\right)\right|
\end{align*}
for every \(p \in \R^\times\). This estimate together with \(\left\lVert h_H\right\rVert_\infty<\infty\) shows that, for \(z,w \in \C_+\), the integrals
\[\int_\R \frac{\kappa\left(ip\right)}{\left(p-z\right)\left(p-w\right)} \,dp \quad \text{and} \quad \int_\R \frac{\oline{\kappa\left(ip\right)}}{\left(p-z\right)\left(p-w\right)} \,dp\]
exist. We have
\begin{equation}
  \label{eq:dag}
\int_\R \frac{\kappa\left(ip\right)}{\left(p-z\right)\left(p-w\right)} \,dp = \int_\R \frac{\kappa\left(-ip\right)}{\left(p+z\right)\left(p+w\right)} \,dp = 0
\end{equation}
because the function
$p \to \frac{\kappa\left(-ip\right)}{\left(p+z\right)\left(p+w\right)}$ 
is holomorphic on $\C_+$.

By the Residue Theorem, for \(z,w \in \C_+\) with \(z \neq w\) and \(\kappa\left(-iz\right) \neq 0 \neq \kappa\left(-iw\right)\), we get
\begin{align*}
\int_\R \frac{\oline{\kappa\left(ip\right)}}{\left(p-z\right)\left(p-w\right)} \,dp &= \int_\R \frac{\kappa\left(-ip\right)}{\left(p-z\right)\left(p-w\right)} \,dp = 2\pi i\left(\frac{\kappa\left(-iz\right)}{z-w} + \frac{\kappa\left(-iw\right)}{w-z}\right)
\\&= 2\pi i \,\frac{\kappa\left(-iz\right)-\kappa\left(-iw\right)}{z-w} 
\ {\buildrel\eqref{eq:kerl-rel} \over =}\ (2\pi)^3 \,Q_H(z,-\overline{w}).
\end{align*}
By  continuity of both sides in \(z\) and \(w\), we get
\begin{equation}
  \label{eq:seconddiffrel}
\int_\R \frac{\oline{\kappa\left(ip\right)}}{\left(p-z\right)\left(p-w\right)} \,dp = (2\pi)^3 \,Q_H(z,-\overline{w})
\quad \mbox{ for every } \quad z,w \in \C_+.  
\end{equation}
For $z,w \in \C_+$, we 
finally obtain 
\begin{align*}
4\pi^2 Q_{H_{h_H}}\left(z,w\right) 
&= 4\pi^2 \la Q_z, h_HRQ_w \ra  
= \int_\R \frac{h_H\left(p\right)}{\left(p-z\right)\left(-p-\oline{w}\right)}\,dp = \int_\R \frac{-h_H\left(p\right)}{\left(p-z\right)\left(p+\oline{w}\right)}\,dp
\\&= \int_\R \frac{-\frac i\pi \cdot \mathrm{Im}\left(\kappa\left(ip\right)\right)}{\left(p-z\right)\left(p+\oline{w}\right)}\,dp = \frac 1{2\pi} \int_\R \frac{\oline{\kappa\left(ip\right)}-\kappa\left(ip\right)}{\left(p-z\right)\left(p+\oline{w}\right)}\,dp
\\& = \frac 1{2\pi} \left(\int_\R \frac{\oline{\kappa\left(ip\right)}}{\left(p-z\right)\left(p+\oline{w}\right)}\,dp - \int_\R \frac{\kappa\left(ip\right)}{\left(p-z\right)\left(p+\oline{w}\right)}\,dp\right)
\\&\ {\buildrel \eqref{eq:dag}\over =}\ \frac 1{2\pi} \int_\R \frac{\oline{\kappa\left(ip\right)}}{\left(p-z\right)\left(p+\oline{w}\right)}\,dp 
\ {\buildrel \eqref{eq:seconddiffrel} \over =}\  4\pi^2  Q_H(z,w)
\end{align*}
This means that the operators $H$ and 
$H_{h_H}$ have the same symbol kernel, hence are equal 
by \cite[Lemma~I.2.4]{Ne99}. 
\end{proof}

\begin{lemma}
Let \(H \neq 0\) be a positive Hankel operator on \(H^2(\C_+)\). 
Then there exist \(c,a \in \R_+\) such that
\[\left|h_H\left(p\right)\right| \geq c \cdot \frac {\left|p\right|}{a^2+p^2}
\quad \mbox{ for every } \quad p \in \R^\times.\] 
\end{lemma}

\begin{proof}
Since \(H \neq 0\), we have \(\mu_H \neq 0\), hence 
\(\mu_H\left(\left(0,a\right]\right)>0\) for some $a > 0$. Then 
setting \(c := \frac{\mu_H\left(\left(0,a\right]\right)}\pi\), for \(p \in \R^\times\), we have
\begin{align*}
\left|h_H\left(p\right)\right| &= \frac 1\pi \int_{\R_+} \frac{\left|p\right|}{\lambda^2+p^2} \,d\mu_H\left(\lambda\right) \geq \frac 1\pi \int_{\left(0,a\right]} \frac{\left|p\right|}{\lambda^2+p^2} \,d\mu_H\left(\lambda\right)
\\&\geq \frac 1\pi \int_{\left(0,a\right]} \frac{\left|p\right|}{a^2+p^2} \,d\mu_H\left(\lambda\right) = c \cdot \frac {\left|p\right|}{a^2+p^2}. \qedhere
\end{align*}
\end{proof}
Choosing the measure \(\mu = \delta_a\) for an \(a \in \R_+\) shows 
that the estimate in this lemma is optimal.

\begin{definition} \mlabel{def:4.3} {(cf. \cite[Thm. 5.13]{RR94})}
A holomorphic function on $\C_+$ 
is called an {\it outer function} if it is of the form
\begin{equation*}
\Out(k,C)(z) = C \exp\left(\frac{1}{\pi i} \int_\R \left[\frac{1}{p-z} - \frac{p}{1 + p^2}\right] \log\left(k\left(p\right)\right) dp\right),
\end{equation*}
where \(C \in \T\) and $k \:\R \to \R_+$  satisfies 
$\int_\R \frac{\left|\log\left(k\left(p\right)\right)\right|}{1 + p^2} \,dp < \infty.$ Then $k = |\Out(k,C)^*|$. 
We write $\Out(k) := \Out(k,1)$. 
If $k_1$ and $k_2$ are two such functions, then 
so is their product, and 
\begin{equation}
  \label{eq:prodout}
 \Out(k_1 k_2) = \Out(k_1) \Out(k_2).
\end{equation}
We also note that the function $k^\vee(p) =k(-p)$ satisfies
\begin{equation}
  \label{eq:outsharp}
 \Out(k^\vee) = \Out(k)^\sharp.
\end{equation}
\end{definition}

\begin{theorem} \mlabel{thm:4.1.5}
Let \(H\) be a positive Hankel operator on \(H^2(\C_+)\). Then, for every \(c \in \R^\times\), we have
\[\delta := h_H + c\textbf{1} \in L^\infty(\R,\C) \quad \text{and} \quad 
\frac{1}{\delta}  \in L^\infty(\R,\C).\]
Further \(H_\delta=H\) and there exists an outer function 
\(g \in H^\infty(\C_+)^\times\) (the unit group of this Banach algebra) 
such that \(\left|g^*\right|^2 = \left|\delta\right|\).
\end{theorem}

\begin{proof}
Since \(h_H\left(\R\right)\subeq i\R\) we have
\[\left\Vert \delta\right\rVert_\infty 
= \sqrt{\left\lVert h_H\right\rVert_\infty^2+c^2} < \infty 
\quad \text{and} \quad \left\lVert \frac{1}{\delta}\right\rVert_\infty \leq \frac 1{|c|},\]
which shows the first statement. 
For the second statement, we notice that 
\(c\textbf{1} \in H^\infty(\C_-)\) implies  \(H_{c\textbf{1}} = 0\) 
by Lemma~\ref{lem:kernelLowerHalfPlane}, so that 
$H_\delta = H_{h_H} + H_{c\textbf{1}} = H + 0 = H$ 
by Lemma~\ref{lem:kernelLowerHalfPlane} and 
Theorem~\ref{thm:CarlesonRepresentant}.

Finally, we have
\[\int_\R \frac{\left|\log \left|\delta\left(p\right)\right|\right|}{1+p^2} \,dp \leq \int_\R \frac{\max\left\{\left|\log \left\Vert \delta\right\rVert_\infty\right|,\left|\log \left\lVert \frac{1}{\delta}\right\rVert_\infty\right|\right\}}{1+p^2} \,dp < \infty\]
and
\[\int_\R \frac{\left|\log \left|\left(\frac 1{\delta}\right)\left(p\right)\right|\right|}{1+p^2} \,dp \leq \int_\R \frac{\max\left\{\left|\log \left\Vert \delta\right\rVert_\infty\right|,\left|\log \left\lVert \frac{1}{\delta}\right\rVert_\infty\right|\right\}}{1+p^2} \,dp < \infty,\]
so we obtain bounded outer functions $\Out(|\delta|^{1/2})$  and 
$\Out(|\delta|^{-1/2})$ 
whose product is $\Out(1) = 1$ (\cite[\S 5.12]{RR94}). 
In particular, $g := \Out(|\delta|^{1/2})$ is invertible 
in  \(H^\infty\left(\C_+\right)\) and $|g^*|^2 = |\delta|$. 
\end{proof}

\subsection{From Hankel positivity to reflection positivity} 
\mlabel{subsec:4.2} 

For a positive Hankel operator \(H\) on \(H^2(\C_+)\) and the corresponding function \(\delta\) from Theorem~\ref{thm:4.1.5}, let \(\nu\) 
be the measure on \(\R\) with 
\[ d\nu\left(x\right) = \left|\delta\left(x\right)\right| \,dx.\] 
As $\delta(-x) = c + h_H(-x) = c - h_H(x) = \oline{\delta(x)}$, 
we have $\delta^\sharp = \delta$, and in particular  
the function $|\delta|$ is symmetric. We consider 
the weighted $L^2$-space 
\(L^2\left(\R,\C,\nu\right)\) with the corresponding scalar product 
$\la \cdot, \cdot \ra_\nu$. 
For the function 
\[ g := \Out(|\delta|^{1/2}) \in H^\infty(\C_+)^\times \] 
 we then have 
\begin{equation}
  \label{eq:g-rels}
 |g^*|^2 = |\delta| \quad \mbox{ and } \quad g^\sharp = g.
\end{equation}
Furthermore, 
$gH^2\left(\C_+\right) = H^2\left(\C_+\right),$ 
and 
\[m_{g^*} : L^2(\R,\nu) \to L^2(\R), \quad f \mapsto g^*\cdot f\]
is an isometric isomorphism of Hilbert spaces. We write 
\[ H^2(\C_+, \nu) := (H^2(\C_+), \|\cdot\|_\nu)\]
for $H^2(\C_+)$, endowed with the scalar product from 
$L^2(\R,\C,\nu)$, so that we obtain a unitary operator 
\[m_g : H^2\left(\C_+,\nu\right) \to H^2(\C_+).\] 
For the unimodular function 
\(u := \frac{\delta}{\left|\delta\right|}\), we get with Theorem~\ref{thm:4.1.5} 
for $a,b \in H^2(\C_+)$:  
\begin{align}
  \label{eq:density1}
 \la a, Hb\ra_{H^2(\C_+)}  
&= \la a^*, \delta R b^*\ra_{L^2(\R)}  
= \la \sqrt{|\delta|}a^*, \sqrt{|\delta|}uR b^*\ra_{L^2(\R)} \notag\\
&= \la a^*, uR b^*\ra_{L^2(\R,\nu)} 
= \la a, H_u b\ra_{H^2(\C_+,\nu)}.
\end{align}
As $\nu$ is symmetric and $u^\sharp = \frac{\delta^\sharp}{|\delta|^\sharp} 
= \frac{\delta}{|\delta|} = u$, 
\[ \theta_u(f)(x) := u(x) f(-x) \] 
defines a unitary involution on $L^2(\R,\nu)$ (and on $L^2(\R)$) 
for which the 
subspace $H^2(\C_+,\nu)$ is $\theta_u$-positive by \eqref{eq:density1} 
(cf. Example~\ref{ex:1.5}). 
Therefore 
\[  (L^2(\R,\nu), H^2(\C_+,\nu), \theta_u, U) \quad \mbox{ with } \quad 
(U_t f)(x) = e^{itx} f(x) \] 
defines a reflection positive one-parameter group.

These are the essential ingredients in the proof of the following theorem: 

\begin{thm} \mlabel{thm:x.1} 
{\rm(Hankel positive representations are reflection positive)} 
Let $(\cE,\cE_+, U, H)$ 
be a regular multiplicity free Hankel 
positive representation of $(\R,\R_+, -\id_\R)$. 
Then there exists an invertible bounded operator $g \in \GL(\cE)$ 
with $g\cE_+ = \cE_+$ commuting with $(U_t)_{t\in \R}$ and a unitary 
involution $\theta \in \GL(\cE)$ such that:
\begin{itemize}
\item[\rm(a)] $\theta U_t \theta = U_{-t}$ for $t \in \R$. 
\item[\rm(b)] $\theta$ is unitary for the scalar product 
$\la \xi, \eta\ra_g :=  \la g \xi, g\eta\ra$. 
\item[\rm(c)] With respect to $\la\cdot,\cdot \ra_g$, the quadruple 
$(\cE,\cE_+,\theta,U)$ is a reflection positive representation. 
\item[\rm(d)] $\la \xi, H \eta \ra 
= \la \xi, \theta \eta \ra_g = \la g\xi, g\theta \eta \ra$ for 
$\xi, \eta \in \cE_+$. 
\end{itemize}
\end{thm} 

\begin{prf} As we have seen in the introduction to Section~\ref{sec:3}, 
the Lax--Phillips Representation Theorem implies that, 
up to unitary equivalence, $\cE= L^2(\R)$ and $\cE_+ = H^2(\C_+)$ 
with $(U_t f)(x) = e^{itx} f(x)$, so that $H$ corresponds to a positive 
Hankel operator on $H^2(\C_+)$. We use the notation from the preceding 
discussion and Theorem~\ref{thm:4.1.5}. Then 
$m_{g^*}$ defines an invertible operator on $L^2(\R)$ commuting with 
$U$, and $\theta := u R$ satisfies (a) and (b). 
Further, (c) and (d) follow from 
\eqref{eq:density1}. 
\end{prf}

\begin{rem} For $H = H_\delta = P_+ \delta R P_+^*$, we 
see with Example~\ref{ex:hankel1}(b) that 
$H m_g^{-2}$ also is a Hankel operator $H_h$ with the operator symbol 
\[ h(x) = \frac{\delta(x)}{g^*(-x)^2} 
= \frac{\delta(x)}{\oline{g^*(x)}^2}.\] 
As $|g^*|^2 = |\delta|$, the function $h$ is unimodular. 
Further $g^\sharp = g$ and $\delta^\sharp = \delta$ imply $h^\sharp = h$, 
so that $\theta_h = h R$ is a unitary involution. 
We think of the factorization 
\[ H = H_h m_g^2 \] 
as a ``polar decomposition'' of $H$.
\end{rem}

\begin{rem} The weighted Hardy space $H^2(\C_+,\nu)$ has the reproducing 
kernel 
\[ Q^\nu(z,w) = \frac{Q(z,w)}{g(z) \oline{g(w)}}.\] 
In fact, for $f \in H^2(\C_+,\nu)$ we have 
\begin{align*}
 \la Q_w^\nu, f \ra_{H^2(\C_+,\nu)} 
&=  f(w) = g(w)^{-1} (fg)(w) = g(w)^{-1} \la Q_w, fg \ra_{H^2(\C_+)} \\
&=g(w)^{-1} \la g^{-1} Q_w, f \ra_{H^2(\C_+,\nu)}.
\end{align*}
\end{rem}

We have a similar result for the symmetric semigroup $(\Z,\N,-\id_\Z)$, 
which corresponds to single unitary operators.

\begin{thm} \mlabel{thm:x.2} 
{\rm(Hankel positive operators  are reflection positive)} 
Let $(\cE,\cE_+, U, H)$ 
be a regular multiplicity free Hankel 
positive operator. 
Then there exists an invertible bounded operator $g \in \GL(\cE)$ 
with $g\cE_+ = \cE_+$ commuting with $U$ and a unitary 
involution $\theta \in \GL(\cE)$ such that:
\begin{itemize}
\item[\rm(a)] $\theta U \theta = U^*$. 
\item[\rm(b)] $\theta$ is unitary for the scalar product 
$\la \xi, \eta\ra_g :=  \la g \xi, g\eta\ra$. 
\item[\rm(c)] With respect to $\la\cdot,\cdot \ra_g$, the quadruple 
$(\cE,\cE_+,\theta,U)$ is a reflection positive operator. 
\item[\rm(d)] $\la \xi, H \eta \ra = \la \xi, \theta \eta \ra_g$ for 
$\xi, \eta \in \cE_+$. 
\end{itemize}
\end{thm} 

\begin{prf} Up to unitary equivalence, we may assume that 
\[ \cE= L^2(\T), \quad \cE_+ = H^2(\bD) \quad \mbox{ with } \quad 
(U f)(z) = z f(z),\] the shift operator (Wold decomposition), so that 
$H$ corresponds to a positive Hankel operator on $H^2(\bD)$. 

Let $C := \Gamma_2 H \Gamma_2^{-1}$ be the corresponding positive 
Hankel operator on $H^2(\C_+)$ (Theorem~\ref{thm:3.5}) 
which we write as $C = H_\delta$ as above in Theorem~\ref{thm:4.1.5}. Then 
\eqref{eq:khrel} in the proof of Theorem~\ref{thm:3.5} 
shows that 
$H = H_k$ for the function $k \: \T \to \C$ defined by 
\[ k \: \T \to \C,  \quad 
k(z) := -\delta(\omega(z)) \oline z \quad \mbox{ for } \quad z \in \T.\] 
Then  $|k(z)| = |\delta(\omega(z))|$ is bounded with a bounded 
inverse. 

We thus find an outer function $g \in H^\infty(\bD)^\times$ 
with $|g^*|^2 = |k|$ and consider the measure 
$d\nu(z) = |k(z)| \, dz$ on $\T$ 
(\cite[Thm. 17.16]{Ru86}; see also Lemma~\ref{lem:b.11}). Then 
\[ m_{g} \: H^2(\bD,\nu) \to H^2(\bD) \] 
is unitary 
and the unimodular function $u := \frac{k}{|k|}$ on $\T$ satisfies, 
for $a,b \in H^2(\bD)$: 
\begin{align}
  \label{eq:density1-disc}
 \la a, Hb\ra_{H^2(\bD)}  
&= \la a^*, k  R b^*\ra_{L^2(\T)}  
= \la \sqrt{|k|}a^*, \sqrt{|k|}uR b^*\ra_{L^2(\T)} \notag\\
&= \la a^*, uR b^*\ra_{L^2(\T,\nu)} 
= \la a, H_u b\ra_{H^2(\bD,\nu)}.
\end{align}
Clearly, $m_{g^*}$ defines an invertible operator on $L^2(\T)$ commuting with 
$U$ and $\theta := u R$ satisfies (a) and (b). 
As in the proof of Theorem~\ref{thm:x.1}, (c) and (d) 
follow from \eqref{eq:density1-disc}.
\end{prf}

\appendix 

\section{Widom's Theorem on Hankel operators 
on the disc} 
\mlabel{app:a} 

In this appendix we recall Widom's classical 
theorem on positive Hankel operators 
on the Hardy space of the unit disc. 
The arguments mostly follow 
Widom's original proof in \cite{Wi66}, including some 
simplifications. In particular the proof for the implication 
(b) $\Rarrow$ (c) was communicated to us by Christian Berg. 
For more information concerning Widom's Theorem, 
we refer to 
\cite[Thm.~B.6.2.1]{Ni02}, which contains in particular the equivalence 
of (a) and the inclusion $H^2(\bD) \subeq L^2([-1,1],\mu)$.

\begin{thm} \mlabel{thm:widom-disc} {\rm(Widom's Theorem; \cite{Wi66})} 
For a finite positive Borel measure $\mu$ on $[-1,1]$
 with moment sequence 
\[ c_j := \int_{-1}^1 x^j\, d\mu(x), \]  
the following are equivalent: 
\begin{itemize}
\item[\rm(a)] The corresponding Hankel operator $H$ 
on $H^2(\bD)$ is bounded, i.e., 
there exists an $\alpha \geq 0$ with 
\[ \la f, H f \ra_{H^2(\bD)} 
=   \int_{-1}^1 |f(x)|^2\, d\mu(x) \leq \alpha \|f\|^2 
\quad \mbox{ for } \quad f \in H^2(\bD).\] 
\item[\rm(b)] $c_j = O(j^{-1})$ for $j \to \infty$.
\item[\rm(c)] 
$\mu([x,1]) = O(1-x)$ as $x \to 1$ and 
$\mu([-1,x]) = O(1+x)$ as $x \to -1$. 
\end{itemize}
\end{thm}

\begin{prf} Since we may decompose 
$\mu = \mu_1 + \mu_2$ with $\mu_1([-1,0)) = 0$ and 
$\mu_2([0,1]) = 0$, we can reduce the discussion to measures 
on $[-1,0]$ and $[0,1]$. In fact, (a) holds for 
$\mu$ if and only if it holds for $\mu_1$ and~$\mu_2$. 
The same is true for (c), where the first condition refers to 
$\mu_1$ and the second one on~$\mu_2$. 
For (b), we write $c_j = c_j^1 + c_j^2$, according to the 
decomposition of $\mu$. If (b) holds for $\mu_1$ and $\mu_2$, 
then it clearly holds for $\mu$. If, conversely, 
(b) holds for $\mu$, then the positive sequence 
$c_{2j} = c_{2j}^1 + c_{2j}^2$ is $O(j^{-1})$, and since both summands 
are positive, we get $c_{2j}^1 = O(j^{-1})$ and 
$c_{2j}^2 = O(j^{-1})$. As the sequences $c_j^1$ and 
$|c_j^2| = (-1)^j c_j^2$ are decreasing, it follows that 
$c_j^1$ and $c_j^2$ are $O(j^{-1})$. 

After this discussion, it suffices to consider the 
case where $\mu =\mu_1$ is a measure on $[0,1]$.

\nin  (b) $\Rarrow$ (c) 
By (b), there exists $\beta > 0$ such that 
\[ \beta/n \geq c_n=\int_0^1 x^nd\mu(x)\geq \int^1_{1-\frac{1}{n}} x^nd\mu(x)\geq
\Big(1 - \frac{1}{n}\Big)^n \mu\Big(\Big[1-\frac{1}{n},1\Big]\Big) \] 
Using that $\big(1-\frac{1}{n}\big)^n\to e^{-1}$ for $n\to\infty$, we find a 
$\gamma > 0$ 
with 
\[ \mu([1-1/n,1])\leq \gamma/n \quad \mbox{ for all } \quad n \in \N.\]
Finally, since $x\mapsto \mu([1-x,1])$ 
is increasing we get $\mu([1-x,1])\leq 2\gamma x$
by choosing $n$ so that $\frac{1}{n+1}< x\leq \frac{1}{n}$.

\nin (c) $\Rarrow$ (b): 
Suppose that $\mu([1-x, 1]) \leq \gamma x$ 
for $x> 0$ sufficiently small. 
Enlarging $\gamma$ if necessary, we may assume that this relation 
holds for all $x\in [0,1]$.
Integration by parts as in Lemma~\ref{lem:intpart} leads for 
$j > 0$ to 
\begin{align*}
c_j &= \int_0^1 x^j\, d\mu(x) 
= \int_0^1 j x^{j-1}\mu([x,1])\, dx 
\leq j \gamma\Big(\frac{1}{j} - \frac{1}{j+1}\Big) 
\leq  \frac{\gamma}{j+1}.  
\end{align*}

\nin (b) $\Rarrow$ (a): Let $\beta > 0$ be such that 
$c_n \leq \beta/(n+1)$ for $n \in \N$. 
For $(a_n)_{n \in \N} \in \ell^2$, we then have 
\[ \Big|\sum_{n,m\geq 0} c_{n+m} \oline{a_n} a_m\Big| 
\leq \sum_{n,m\geq 0} c_{n+m} |a_n| |a_m| 
\leq \beta \sum_{n,m\geq 0} \frac{|a_n| |a_m|}{1 + n + m}
\leq \beta \pi \|a\|^2 \] 
by Hilbert's Theorem (\cite[Part B, 1.6.7]{Ni02}).

\nin (a) $\Rarrow$ (c): For $0 < r < 1$, we consider the function 
\[ f(z) := \sum_{j = 0}^\infty r^j z^j = \frac{1}{1 - r z} \] 
in $H^2(\bD)$. Then 
\[ H(f,f) = \int_{-1}^1 |f(x)|^2\, d\mu(x) 
\leq \|H\| \|f\|^2 = \|H\| \frac{1}{1 -r^2}.\]
This leads to the estimate 
\[ \frac{\mu([r,1]}{(1-r^2)^2} 
\leq \int_r^1 \frac{1}{(1-rx)^2} \, d\mu(x) 
\leq \int_r^1 |f(x)|^2\, d\mu(x) 
\leq \|H\| \|f\|^2 = \frac{\|H\|}{1 -r^2}\]
and further to 
\[ \mu([r,1]) \leq 2 \|H\| (1-r),\] 
which implies (c). 
\end{prf}

\begin{defn} A measure $\mu$ on $\bD$ for which all 
functions in $H^2(\bD)$ are square-integrable is called 
a {\it Carleson measure} (cf.\ \cite{Ca62}, \cite[p.~327]{Ni02}). 
The implication (c) $\Rarrow$ (a) in Widom's Theorem 
also follows from the much more general Theorem 1 in \cite{Ca62} 
concerning measures on the disc. 
\end{defn}

\begin{ex} \mlabel{ex:hilb} 
(a) For the Lebesgue measure $d\mu(x)= dx$ on $[0,1]$, we obtain the 
moment sequence $c_j = \frac{1}{j+1}$, and 
by Hilbert's Theorem (\cite[Part B, 1.6.7]{Ni02}), 
the corresponding Hankel operator is 
bounded. In particular, there exists a constant 
$C$ with 
\[ \int_{[0,1]}|f(x)|^2 \, dx \leq C \|f\|^2 \quad \mbox{ for } \quad 
f \in H^2(\bD).\] 
Note that $\int_{-1}^1 \frac{dx}{1 -x^2} = \infty.$ 

\nin (b) If $s > -\frac{1}{2}$, then the measure 
$d\mu(x) = x^s\, dx$ on $(0,1)$ is finite with moment sequence 
\[ c_j = \int_0^1 x^{j+s}\, dx 
= \frac{1}{j + 1 + s}.\] 
\end{ex}

\begin{ex} Suppose that $\mu$ is a measure on $(-1,1)$ with 
$\int_{-1}^1 \frac{d\mu(x)}{1 -x^2} < \infty.$ 
On $H^2(\bD)$, the Szeg\"o kernel is given by 
\[ Q_w(z) = Q(z,w) = \frac{1}{2\pi} \frac{1}{1 - z \oline w} \] 
We therefore have 
\[ |f(w)|^2 \leq \|f\|^2 \|Q_w\|^2 
= \|f\|^2 Q(w,w) = \frac{1}{2\pi} \frac{\|f\|^2}{1 - |w|^2},\] 
and this shows that 
\[ \int_{-1}^1 |f(x)|^2 \, d\mu(x) 
\leq \|f\|^2 \int_{-1}^1 \frac{d\mu(x)}{1-x^2} < \infty.\]
\end{ex}

For the sake of easier reference, we include the following 
version of integration by parts in this appendix. 
\begin{lem} {\rm(Integration by parts)} 
\mlabel{lem:intpart}
Let $a < b$ be real numbers and $f \in C^1([a,b])$. 
For a finite positive Borel measure $\mu$ on $[a,b]$ we then have 
\[ \int_a^b f(x)\, d\mu(x) 
= \mu([a,b]) f(a) + \int_a^b \mu([t,b]) f'(t)\, dt\] 
and 
\[ \int_a^b f(x)\, d\mu(x) 
= \mu([a,b]) f(b) - \int_a^b \mu([a,t]) f'(t)\, dt\] 
\end{lem}

\begin{prf} With Fubini's Theorem, we obtain 
  \begin{align*}
\int_a^b f(x)\, d\mu(x) 
&= \int_a^b\Big( f(a) + \int_a^x f'(t)\, dt\Big)\, d\mu(x) 
= f(a) \mu([a,b]) +\int \int_{a \leq t \leq x \leq b} \, d\mu(x)\, f'(t)\, dt\\ 
&= f(a) \mu([a,b]) +\int_a^b \mu([t,b])\, f'(t)\, dt.
\end{align*}
We likewise get the second assertion.
\end{prf}

\begin{rem} There exists a refinement of Widom's Theorem 
characterizing those Hankel operators for which the 
measure $\mu$ lives on $[0,1)$, i.e., $\mu((-1,0)) = 0$ 
(\cite{GP15}). This condition means that, 
not only the moment sequence 
$(c_n)_{n \in \N_0}$ of $\mu$ is positive definite on $\N_0$, 
but also the shifted sequence $(c_{n+1})_{n \in \N_0}$. 
As the shifted sequence satisfies 
\[ c_{n+1} = \int_{-1}^1 t^n \cdot t\, d\mu(t) \quad \mbox{ for } \quad 
n \in \N_0,\] 
its positive definiteness is equivalent to the positivity of the 
measure $t\, d\mu(t)$, which is equivalent to $\mu((-1,0)) = 0$. 
\end{rem}

\section{The Banach $*$-algebra $(H^\infty(\Omega), \sharp)$} 
\mlabel{app:b} 

Let $\Omega \subeq \C$ be a proper simply connected domain. 
By the Riemann Mapping Theorem, there exists a biholomorphic 
map $\phi \: \bD \to \Omega$, so that 
$\sigma(z) := \phi(\oline{\phi^{-1}(z)})$ defines an antiholomorphic involution 
on $\Omega$. We thus obtain on the Banach algebra $H^\infty(\Omega)$ of bounded 
holomorphic functions on $\Omega$ the isometric antilinear involution 
\begin{equation}
  \label{eq:hinftysharp}
 f^\sharp(z) := \oline{f(\sigma(z))},
\end{equation}
turning into a Banach $*$-algebra. 
As this algebra and some of its subsemigroups play a key role in many of 
our arguments, we take in this appendix a closer look at some of its 
features. Its natural weak topology is of utmost importance 
because the weakly continuous positive functionals turn out to be closely 
related to Hankel operators resp., to measures on the fixed point set 
$\Omega^\sigma$ of $\sigma$ on $\Omega$ 
(cf.\ Proposition~\ref{prop:4.6}).

\subsection{The weak topology} 
\mlabel{app:b.1} 

\begin{lem} \mlabel{lem:b.0} Two antiholomorphic involutions on $\Omega$ 
are conjugate under the group  $\Aut(\Omega)$ of biholomorphic automorphisms.   
\end{lem}

\begin{prf} By  the Riemann Mapping Theorem, we may assume that 
$\Omega = \bD$ is the unit disc. Let $\sigma \: \bD \to\bD$ be an antiholomorphic 
involution. Then $\sigma$ is an isometry for the hyperbolic metric. 
Therefore the midpoint of $0$ and $\sigma(0)$ is fixed by $\sigma$. Conjugating 
by a suitable automorphism of~$\bD$, we may therefore assume that 
$\sigma(0) = 0$. Then $\psi(z) := \sigma(\oline z)$ is a holomorphic automorphism 
fixing $0$, hence of the form $\psi(z) = e^{i\theta} z$ for some $\theta \in \R$,
so that $\sigma(z) = e^{i\theta} \oline z = \gamma(\oline{\gamma^{-1}(z)})$ 
for $\gamma(z) = e^{i\theta/2}z$. 
\end{prf}

As all these involutions are conjugate under the group $\Aut(\Omega)$ 
by Lemma~\ref{lem:b.0}, all Banach $*$-algebras $(H^\infty(\Omega), \sharp)$ 
are isomorphic. 

According to Ando's Theorem 
(\cite{An78}), the Banach space $H^\infty(\Omega) \cong H^\infty(\bD)$ 
has a unique predual space $H^\infty(\Omega)_* \subeq  H^\infty(\Omega)^*$, 
hence carries a natural {\it weak topology}, which is the initial 
(locally convex) topology defined by the elements of the predual.  
Note that the predual is norm-closed in $H^\infty(\Omega)^*$ because its 
embedding is isometric. 

\begin{ex} \mlabel{ex:4.1} (a) (The upper half plane $\C_+$) We consider on 
$H^\infty(\C_+)$ the  continuous linear functionals 
\[ \eta_f(g) := \int_\R g^*(x)f(x)\, dx, \qquad f \in L^1(\R), g \in H^\infty(\C_+).\] 
Recall that $L^\infty(\R) \cong L^1(\R)^*$. By 
\cite[Ex.~12, p.~115]{RR94}, the closed subspace 
\[ H^\infty(\C_+) 
\cong \{ g \in L^\infty(\R) \: g H^2(\C_+) \subeq H^2(\C_+) \} \] 
of $L^\infty(\R)$ coincides with the annihilator 
of the subspace $H^1(\C_+)$ of $L^1(\R)$. 
Therefore 
\[ (L^1(\R)/H^1(\C_+))^* \cong H^1(\C_+)^\bot \cap L^\infty(\R) = H^\infty(\C_+).\]
By Ando's Theorem, 
\begin{equation}
  \label{eq:quot}
H^\infty(\C_+)_* \cong L^1(\R)/H^1(\C_+) 
\end{equation}
is the unique predual of $H^\infty(\C_+)$. 
In particular, the weak topology is the initial topology with respect to the functionals 
$\eta_f$, $f \in L^1(\R)$. 

As the predual $H^\infty(\C_+)_*$ is a norm-closed subspace of $H^\infty(\C)^*$, 
the image of the map 
\[ L^1(\R) \to H^\infty(\C_+)^*, \quad f \mapsto \eta_f \] 
is closed. 
For $f \in L^1(\R)$, we have $\eta_f^\sharp = \eta_{f^\sharp}$, so 
that $\eta_f$ is symmetric if $f = f^\sharp$. 

\nin (b) (The unit disc $\bD$) 
For the disc, we define for $f \in L^1(\T)$ the functional 
\[ \eta_f(g) 
= \int_{\T} f(e^{it}) g^*(e^{it})\, dt  \] 
on $H^\infty(\bD)$. Then the unique predual of $H^\infty(\bD)$ is the quotient of 
$L^1(\T)$ by the subspace \break 
$\{ f \in L^1(\T) \: \eta_f = 0 \},$ 
which by \cite[Ch.~17, Ex.~2.9]{Ru86} 
is contained in $H^1(\bD)$ 
(see also the proof of Lemma~\ref{lem:polweakdense}). 
For $g \in H^\infty(\bD)$ and $f \in H^1(\bD)$, we have 
\[ \frac{\eta_{f^*}(g)}{2\pi} 
= \int_{\T} f^*(e^{it}) g^*(e^{it})\, \frac{dt}{2\pi}  
= (fg)(0) = f(0) g(0), \] 
so that $\eta_{f^*} = 0$ is equivalent to $0 = f(0) = \frac{\eta_{f^*}(1)}{2\pi}$. 
With $H^1_0(\bD) := \{ f \in H^1(\bD) \: f(0) = 0\},$ 
we thus obtain 
\begin{equation}
  \label{eq:predual-disc}
 H^\infty(\bD)_* \cong L^1(\T)/H^1_0(\bD).
\end{equation}
\end{ex}

\begin{lem}
  \mlabel{lem:unitball} 
On the closed unit ball $B \subeq H^\infty(\Omega)$, 
the following topologies coincide and turn $B$ into a compact space: 
\begin{itemize}
\item[\rm(a)] The topology $\tau_c$ of uniform convergence on compact subsets of $\Omega$. 
\item[\rm(b)] The topology $\tau_p$ of pointwise convergence.
\item[\rm(c)] The weak topology $\tau_w$. 
\end{itemize}
\end{lem}

\begin{prf} By Montel's Theorem, $(B,\tau_c)$ is a compact space. 
Since $\tau_p$ is Hausdorff and $(B,\tau_c) \to (B,\tau_p)$ is continuous, 
the compactness of $(B,\tau_c)$ implies that $\tau_c = \tau_p$. 

To show that $\tau_w = \tau_p$, 
we may w.l.o.g.\ assume that $\Omega = \C_+$. 
For each $z \in \C_+$ and $g \in H^\infty(\C_+)$, we have 
\[ g(z) = \int_\R P_z(x)g^*(x)\, dx,\] 
where $P_z(x) = P(z,x)$ is the Poisson kernel of $\C_+$. As the functions 
$P_z$ are $L^1$, it follows that point evaluations are weakly continuous. 
Therefore the map $(B,\tau_w) \to (B,\tau_p)$ is continuous. 
As $(B,\tau_w)$ is compact by the Banach--Alaoglu Theorem, this map is a homeomorphism, 
and thus $\tau_w = \tau_p$. 
\end{prf}

\begin{rem}  (a) For a $\sigma$-finite measure space 
$(X,\fS,\mu)$, the unique predual of $L^\infty(X,\fS,\mu)$ is 
the space $L^1(X,\fS,\mu)$ (Grothendieck, \cite{Gr55}). However, the space  
$L^\infty(X,\fS,\mu)$ only depends on the measure class $[\mu]$. 
From this perspective, one should think of its predual as the space 
$\{ f \mu \: f \in L^1(X,\fS,\mu)\}$ of all finite measures on 
$(X,\fS)$ which are absolutely continuous with respect to~$\mu$. 

With this observation, it is clear how to identify the predual 
of $H^\infty(\C_+)$ in terms for weighted Hardy spaces. 
In particular, $H^\infty(\C_+)_* \cong L^1(\R,w\, dx)/H^1(\R,w\, dx)$ for any 
positive measurable function $w \: \R \to \R_+$.

\nin (b) (Saks spaces) 
The theory of Saks spaces, i.e., Banach spaces $E$ with an additional 
locally convex topology $\gamma$ satisfying certain compatibility conditions 
is a natural context to deal with similar structures. 
We refer to Cooper's monograph \cite{Co87} for a detailed exposition
of this theory. We shall not need it here. An interesting result 
one finds in \cite[Prop.~V.3.2]{Co87} 
is that the space of continuous homomorphisms $(H^\infty(\bD),\beta) \to \C$,
where $\beta$ is the topology on $H^\infty(\bD)$ defined by the Saks space 
structure, is homeomorphic to $\bD$ (the point evaluations). 
\end{rem}

\begin{prop} The multiplication on $H^\infty(\Omega)$ is separately 
continuous with respect to the weak topology, i.e., the 
multipication maps $m_g(f) = gf$ are weakly continuous. 
Moreover, the involution $\sharp$ is weakly continuous. 
\end{prop}

\begin{prf} It suffices to verify this for $\Omega = \C_+$. In this 
case it follows from 
\[ \eta_f(g^\sharp) = \oline{\eta_{f^\sharp}(g)} \quad \mbox{ and } \quad 
\eta_f(gh) = \eta_{f g^*}(h) \quad \mbox{ for } \quad 
f \in L^1(\R), g,h \in H^\infty(\C_+).\qedhere\]
\end{prf}

\subsection{Subsemigroups spanning weakly dense subalgebras} 
\mlabel{app:b.2}

For $\Omega = \bD$ we have $\sigma(z) = \oline z$, so that 
\begin{equation}
  \label{eq:sharpeq1}
f^\sharp(z) = \oline{f(\oline z)}\quad \mbox{ for } \quad f \in H^\infty(\bD).
\end{equation}
The elements $(z^n)_{n \geq 0}$ define 
a cyclic subsemigroup of $H^\infty(\bD)$ consisting 
of $\sharp$-symmetric elements and, for $f \in L^1(\T)$, we have 
\begin{equation}
  \label{eq:lambdafoutra-disc}
 \eta_f(z^n) = \int_\T e^{int} f(e^{it})\, dt = \hat f(-n) 
\quad \mbox{ for } \quad n \in \N_0.
\end{equation}
For $\Omega = \C_+$ we have $\sigma(z) = - \oline z$, so that 
\begin{equation}
  \label{eq:sharpeq2}
f^\sharp(z) = \oline{f(-\oline z)}\quad \mbox{ for } \quad f \in H^\infty(\C_+).
\end{equation}
The elements $(e_{it})_{t > 0}$ define 
a one-parameter semigroup of $H^\infty(\C_+)$ consisting 
of $\sharp$-symmetric elements and, for $f \in L^1(\R)$, we have 
\begin{equation}
  \label{eq:lambdafoutra}
 \eta_f(e_{it}) = \int_\R e^{itx} f(x)\, dx 
= \hat f(-t) \quad \mbox{ for } \quad t \geq 0.
\end{equation}

\begin{lem} \mlabel{lem:polweakdense} 
{\rm(The Density Lemma)} 
  \begin{itemize}
  \item[\rm(a)] The polynomials $\C[z] \subeq H^\infty(\bD)$ 
are dense with respect to the weak topology. 
  \item[\rm(b)] The one-parameter semigroup 
$(e_{it})_{t > 0}$ spans a weakly dense subspace of $H^\infty(\C_+)$. 
\item[\rm(c)] For the strip 
$\bS_\beta = \{ z \in \C \: 0 < \Im z < \beta\}$, 
the functions $(e_{it})_{t \in \R}$ 
 span a weakly dense subspace of $H^\infty(\bS_\beta)$. 
  \end{itemize}
\end{lem}

\begin{prf} (a)  (\cite[Prop.~V.2.2]{Co87}) 
If $f \in L^1(\T)$ is such that $\eta_f$ vanishes on all polynomials, 
then all negative Fourier coefficients of $f$ vanish:
\[ \hat f(-n) = \int_0^{2\pi} f(e^{it}) e^{int}\, dt = 0 \quad \mbox{ for } \quad 
n > 0,\] 
and \cite[Ch.~17, Ex.~2.9]{Ru86} 
implies that $f \in H^1(\bD)$. 
Now the vanishing of $\eta_f$ follows from 
\[ \eta_f(h) 
= \int_{\T} f(e^{it}) h(e^{it})\, dt 
= 2\pi (fh)(0) = 2 \pi f(0) h(0) = 0, \] 
because $2\pi f(0) = \eta_f(1) = 0$. 

\nin (b) We have to show that, if a functional 
$\eta_f$, $f \in L^1(\R)$, vanishes on each $e_{it}, t > 0$, then 
$\eta_f = 0$. So suppose that $\eta_f(e_{it}) = \hat f(-t) = 0$ for $t > 0$. 
We claim that this implies that 
\begin{equation}
  \label{eq:cauchvan}
  \int_\R \frac{f(t)}{t-z}\, dt = 0 \quad \mbox{ for } \quad 
\Im z < 0.
\end{equation}
In fact, for $\Im z < 0$, we have 
\[ \frac{1}{t-z} = -i \int_0^\infty e^{itx} e^{-ixz}\, dx. \] 
We thus obtain 
\[ \int_\R \frac{f(t)}{t-z}\, dt 
= -i \int_0^\infty \int_\R e^{itx} e^{-ixz} f(t)\, dt\, dx 
= -i \int_0^\infty \hat f(-x)e^{-ixz}\, dx = 0.\] 
In view of \cite[Thm.~5.19(ii)]{RR94}, 
\eqref{eq:cauchvan} implies that $f \in H^1(\C_+)$ in the  sense 
that $f$ is the boundary value of an $H^1$-function on $\C_+$.  

We now show that this implies $\eta_f = 0$.
In fact, for $h \in H^\infty(\C_+)$, we obtain 
\[ \eta_f(h) = \int_\R f(x)h(x)\, dx = 0\] 
because the function $fh \in H^1(\C_+)$ has a continuous Fourier transform 
vanishing on $\R_-$, hence also in $0$. 
\begin{footnote}
{In \cite[Ex.~12, p.~115]{RR94} one finds the interesting characterization 
that, for $1 \leq p,q \leq \infty$ and $p^{-1} + q^{-1} = 1$, a 
function $f \in L^p(\R)$ is contained in $H^p(\C_+)$ if and only if 
$\eta_f$ vanishes on $H^q(\C_+)$.}
\end{footnote}

\nin (c) Let $f = (f_0, f_1) \in L^1(\R) \oplus L^1(\R)$ be such 
that $\eta_f(e_{it}) = 0$ for every $t \in\R$. These numbers evaluate to 
\begin{align*}
\eta_f(e_{it}) 
&= \int_{\R} e^{itx} f_0(x)\, dx + \int_{\R} e^{it(x + i\beta)} f_1(x)\, dx 
= \int_{\R} e^{itx} f_0(x)\, dx + e^{-t\beta} \int_{\R} e^{itx} f_1(x)\, dx
\\ 
&= \hat{f_0}(-t) + e^{-t\beta} \hat{f_1}(-t).
\end{align*}
We thus arrive at the relation 
\begin{equation}
  \label{eq:bwrel}
 \hat{f_1}(t) = - e^{-t\beta} \hat{f_0}(t) \quad \mbox{ for } \quad t \in \R, 
\quad \mbox{ resp.} \quad \hat{f_1} = - e_{-\beta} \hat{f_0}.
\end{equation}

Let $\cE \subeq L^1(\R) \times L^1(\R)$ be the closed linear subspace of all pairs 
$(g_0, g_1)$ satisfying $\hat g_1 = - e_{-\beta} \hat g_0$. This is a closed 
subspace invariant under the translation action 
$\alpha_s(g) = g(\cdot + s)$. In the Banach algebra 
$L^1(\R)$, we consider the approximate identity 
\[ \delta_n(x) := \frac{n}{\sqrt{2\pi}} e^{-\frac{n^2x^2}{2}}.\] 
The pairs $(\alpha_s(\delta_n*g_0), \alpha_s(\delta_n*g_1))$ for $n\in\N$ and $(g_0,g_1)\in\cE$ extend to pairs of holomorphic maps $\C \to L^1(\R)$, given concretely by the functions 
\[ \alpha_z(\delta_n*g_i)(x) = \delta_n*g_i(x+z):= \frac{n}{\sqrt{2\pi}} \int_{\R}e^{-\frac{n(x+z-t)^2}{2}}g_i(t)dt
\quad \mbox{ for } \quad  z \in \C,\;i=0,1.\] 
For any such pair, the Fourier transform of $\alpha_z(\delta_n*g_0)=\delta_n*g_0(\cdot + z)$ is $e_{iz} \hat{\delta_n}\hat{g_0}$. 
For $z = \beta i$, this function coincides with $-\hat{\delta_n}\hat{g_1}$, so that the injectivity 
of the Fourier transform leads to 
\[ \delta_n*g_1 = - \alpha_{\beta i}(\delta_n*g_0).\]
As $\lim_{n \to \infty} \delta_n * h  = h$ for any $h \in L^1(\R)$, it follows that the set of pairs $(\delta_n*g_0,\delta_n*g_1)$ is dense in $\cE$.

Let $h \in H^\infty(\bS_\beta)$. Then 
$h_z(x) := h(z + x)$ defines a bounded weakly holomorphic family in $L^\infty(\R)$, i.e., the function $\bS_\beta\ni z\mapsto\eta_f(h_z)\in\C$ is holomorphic for every $f\in L^1(\R)$.
For $\delta_n*g = (\delta_n*g_0, \delta_n*g_1) \in \cE$ as above, the function 
\[ \gamma \: \bS_\beta \to \C, \quad 
\gamma(z) := \int_\R \alpha_z(\delta_n*g_0)(x) h(z + x)\, dx \] 
is holomorphic and bounded on $\bS_\beta$ and extends continuously to the closed strip. 
Moreover, its lower boundary values are constant because of the translation 
invariance of Lebesgue measure. Therefore $\gamma$ is constant and we obtain 
in particular 
\begin{align*}
\int_\R \delta_n*g_0(x) h^*(x)\, dx 
&= \gamma(0) = \gamma(\beta i) 
= \int_\R \alpha_{\beta i}(\delta_n*g_0)(x) h^*(\beta i + x)\, dx 
\\&= -\int_\R \delta_n*g_1(x) h^*(\beta i + x)\, dx,
\end{align*}
which means that $\eta_{\delta_n*g}(h) = 0$. 
With the density argument from above, this entails that 
$\eta_f(h) = 0$ for each $h \in H^\infty(\bS_\beta)$, and $f\in\cE$.
This proves that the functions $(e_{it})_{t \in \R}$ span a weakly 
dense subspace of $H^\infty(\bS_\beta)$.
\end{prf}

\subsection{Weakly continuous positive functionals} 
\mlabel{app:b.3} 

The compact space of  characters of the commutative 
Banach algebra $H^\infty(\Omega)$ is a complicated space 
in which the evaluation functionals 
$\delta_z(f) = f(z)$, $z \in \Omega$, are dense by 
Carleson's Corona Theorem (\cite{Ca62}). 
These characters satisfy $\delta_z^\sharp(f) := \oline{f^\sharp(z)}
= \delta_{\sigma(z)}(f)$, so that $\delta_z^\sharp = \delta_z$ is equivalent to 
$z \in \Omega^\sigma$. 
The following proposition shows that this construction 
exhausts the set of weakly continuous $*$-characters. 

\begin{prop} 
\mlabel{prop:4.5} 
The weakly continuous $*$-homomorphisms 
$(H^\infty(\Omega),\sharp) \to \C$ are the maps 
\[ \delta_\lambda(f) := f(\lambda) \quad \mbox{ for } \quad 
\lambda \in \Omega^\sigma = \{ z \in \Omega \:  \sigma(z) = z\}.\] 
\end{prop}

\begin{prf} In view of the Riemann Mapping Theorem, we may w.l.o.g.\ 
assume that $\Omega = \C_+$ with $\sigma(z) = -\oline z$ 
(Lemma~\ref{lem:b.0}). 
Clearly, each $\delta_\lambda$ defines a weakly continuous $*$-homomorphism. 
Suppose, conversely, that $\chi \: (H^\infty(\C_+),\sharp) \to \C$ 
is a weakly continuous unital $*$-homomorphism. 
As $(e_{it})_{t > 0}$ is an involutive subsemigroup 
spanning a weakly dense subspace, $\chi$ is 
uniquely determined by its values on this semigroup. 
This defines a continuous non-zero homomorphism 
\[ \R_+ \to ([0,1],\cdot), \quad 
t \mapsto \chi(e_{it}), \] 
hence is of the form 
$t \mapsto e^{-t\lambda} = e_{it}(i\lambda)$ for some $\lambda \geq 0$. 
Writing $\chi = \eta_f$ for some $f \in L^1(\R)$, we see that 
$\eta_f(e_{it}) = \hat f(-t)$ tends to $0$ for $t \to \infty$ 
(Riemann--Lebesgue Lemma), 
so that we must have $\lambda >0$ and thus $\chi = \delta_{i\lambda}$. 
%
\end{prf}

\begin{prop} \mlabel{prop:4.6} 
The weakly continuous positive functionals 
$(H^\infty(\Omega),\sharp) \to \C$ are the maps 
\[ \eta_\mu(f) := \int_{\Omega^\sigma} f(\lambda)\, d\mu(\lambda),\] 
where $\mu$ is a finite positive Borel measure on $\Omega^\sigma$. 
\end{prop}

\begin{prf} Again, we may w.l.o.g.\ assume that $\Omega = \C_+$ with 
$\sigma(z) = - \oline z$. 
For the elements $e_{it}$, $t > 0$, of $H^\infty(\C_+)$, 
the reproducing property of the Poisson kernel 
\[  P(z,x) = P_z(x) = \frac{1}{\pi} \frac{\Im z}{|z-x|^2} \] 
(see \eqref{eq:poissker} below) leads to 
\begin{equation}
  \label{eq:repofu}
e^{-t\lambda} = e_{it}(i\lambda) 
= \frac{1}{\pi} \int_\R e^{it x} \frac{\lambda}{\lambda^2 + x^2}\, dx.
\end{equation}
If $\mu$ is a finite positive measure on $\R_+ = (0,\infty)$, then 
\eqref{eq:repofu} shows that the function 
\begin{equation}
  \label{eq:psimu}
 \psi_\mu(x) 
:= \frac{1}{\pi} \int_0^\infty 
\frac{\lambda}{\lambda^2 + x^2}\, d\mu(\lambda)
= \int_0^\infty P(i\lambda,x)\, d\mu(\lambda)
\end{equation}
is $L^1$ on $\R$ with total integral $\mu((0,\infty))$. 
For $f \in H^\infty(\C_+)$, we obtain 
\begin{equation}
  \label{eq:mueval}
\eta_{\psi_\mu}(f) 
= \int_\R \psi_\mu(x) f^*(x)\, dx 
= \int_0^\infty \Big(\int_\R P_{i\lambda}(x)f^*(x)\, dx\Big)\, d\mu(\lambda) 
= \int_0^\infty f(i\lambda)\, d\mu(\lambda).
\end{equation}
Therefore all functionals $\eta_{\mu}$ are weakly continuous. 
That they are positive follows from 
\[ \eta_{\psi_\mu}(f^\sharp f) 
= \int_0^\infty \oline{f(i\lambda)} f(i\lambda)\, d\mu(\lambda).
= \int_0^\infty |f(i\lambda)|^2\, d\mu(\lambda).\]

Suppose, conversely, that $\eta_f \: H^\infty(\C_+) \to \R$ 
is a weakly continuous positive functional. As 
the semigroup $(e_{it})_{t \geq 0}$ spans a weakly dense subspace, $\eta_f$ 
is determined uniquely by its restriction to this semigroup, 
on which it defines a continuous bounded positive definite function. 
All these functions are Laplace transforms 
$\cL(\mu)$ of a finite positive Borel measure on $[0,\infty)$ 
by the  
Hausdorff--Bernstein--Widder Theorem (\cite[Thm.~6.5.12]{BCR84}), 
so that it remains to show that $\mu(\{0\}) = 0$. This follows from 
\[ 0 
= \lim_{t \to \infty} \hat f(-t) 
= \lim_{t \to \infty} \eta_f(e_{it}) 
= \lim_{t \to \infty} \cL(\mu)(t) 
= \mu(\{0\}) +  \lim_{t \to \infty} \int_0^\infty e^{-t\lambda}\, d\mu(\lambda)
= \mu(\{0\}).\qedhere \] 
\end{prf}

\subsection{Weakly continuous representations} 
\mlabel{app:b.4} 

\begin{prop} \mlabel{prop:4.9} For a $*$-representation $(\pi,\cH)$ of 
the Banach $*$-algebra $(H^\infty(\Omega), \sharp)$, the following are equivalent: 
\begin{itemize}
\item[\rm(a)] For every trace class operator $A \in B_1(\cH)$, the functional 
$\pi^A(f) := \tr(A \pi(f))$ is weakly continuous. 
\item[\rm(b)] For every $\xi \in \cH$, the matrix coefficient 
$\pi^\xi(f) := \la \xi, \pi(f) \xi \ra$ is weakly continuous. 
\item[\rm(c)] There exists a dense subspace $\cD \subeq \cH$ such that, 
for every $\xi \in \cD$, the matrix coefficient 
$\pi^\xi(f) := \la \xi, \pi(f) \xi \ra$ is weakly continuous. 
\end{itemize}
\end{prop}

\begin{prf} Clearly (a) $\Rarrow$ (b) $\Rarrow$ (c). 
It remains to shows that (c) implies (a). 
To this end, assume (c). Then each $\pi^\xi$, $\xi \in \cD$, is a 
weakly continuous positive functional, hence satisfies 
\[ |\pi^\xi(f)| \leq  \pi^\xi(\1) \|f\| = \|\xi\|^2 \|f\| \] 
by \cite[Prop.~2.1.4]{Dix64}. This implies in particular that 
\[  \|\pi(f)\xi\|^2 = |\pi^\xi(f^\sharp f)| 
\leq  \|\xi\|^2 \|f^\sharp f\|  \leq  \|\xi\|^2 \|f\|^2, \] 
so that $\|\pi\| \leq 1$. We now consider the map 
\[ \pi^* \: B_1(\cH) \to H^\infty(\Omega)^*, \quad 
\pi^*(A)(f) := \tr(A \pi(f)) \] 
which is a linear contraction. As the map 
\[ \cH \times \cH \to B_1(\cH), \quad (\xi, \eta) \mapsto P_{\xi,\eta}, \quad 
P_{\xi,\eta}(v) := \la \xi, v \ra \eta \] 
is sesquilinear and continuous, the Polarization Identity implies 
that $\pi^*(\{ P_{\xi,\eta} \: \xi,\eta \in \cD\})$ 
consists of weakly continuous functionals. 
Now (a) follows from the norm-closedness of the predual 
$H^\infty(\Omega)_*$ in $H^\infty(\Omega)^*$ (cf.~Example~\ref{ex:4.1}(a)), 
the norm-continuity of $\pi^*$, and the density 
of the span of $P_{\xi,\eta}$, $\xi, \eta \in \cD$, in $B_1(\cH)$. 
\end{prf}

\begin{defn} Representations of the $*$-algebra $(H^\infty(\Omega),\sharp)$ 
satisfying the equivalent conditions 
in Proposition~\ref{prop:4.9} are called {\it weakly continuous}. 
\end{defn}

From weakly continuous positive functionals, we 
obtain weakly continuous cyclic $*$-rep\-resen\-tations 
of $(H^\infty(\Omega),\sharp)$ by Proposition~\ref{prop:4.9}(c). 
Another source of such representations 
are positive Hankel operators for the multiplication representation 
of $(H^\infty(\Omega),\sharp)$ on $H^2(\Omega)$ 
(Proposition~\ref{prop:1.6}). 
For the sake of easier reference, we formulate 
the corresponding result for $\C_+$ explicitly:

\begin{prop} \mlabel{prop:4.11} Suppose that $H^2(\C_+)$ is $\theta_h$-positive 
for the operator $\theta_h(f)(x) = h(x)f(-x)$ 
and $h \in L^\infty(\R)$ satisfying $h^\sharp = h$. 
Let $\cH$ denote the Hilbert space defined by the positive semidefinite 
form $\la f,g \ra_h := \la f, \theta_h g \ra$ on $H^2(\C_+)$ 
and write $q \: H^2(\C_+)  \to \cH$ for the natural map 
with dense range. Then there exists a 
$*$-representation $(\pi, \cH)$ of the Banach $*$-algebra 
$(H^\infty(\C_+),\sharp)$ which is uniquely determined by the relation 
\begin{equation}
  \label{eq:ostrafo}
q(fg) = \pi(f)q(g)  \quad \mbox{ for } 
\quad f \in H^\infty(\C_+), g \in H^2(\C_+).
\end{equation}
\end{prop}

\begin{prf} In view of 
Theorem~\ref{thm:3.5}(b), the operator 
$H_h := P_+ \theta_h P_+^*$ is a positive Hankel operator 
for the representation of the $*$-algebra 
$(H^\infty(\C_+),\sharp)$ on $H^2(\C_+)$. 
Hence the assertion follows from Proposition~\ref{prop:1.6}. 
\end{prf}

\begin{rem} In the context of Proposition~\ref{prop:4.11}, 
we have for $g \in H^2(\C_+)$ 
\begin{align}\label{eq:4.12}
\la q(g), \pi(f) q(g) \ra_{\hat\cE} 
&=  \la g, \theta_h(fg) \ra 
= \int_\R \oline{g(x)} g(-x) h(x) f(-x)\, dx \notag \\
&= \int_\R g^\sharp(x) g(x) h(-x) f(x)\, dx 
= \eta_{h^\vee g^\sharp g}(f), 
\end{align}
where we use the notation 
$h^\vee(x) := h(-x)$ for $x \in \R$.
These positive functionals are weakly continuous on $H^\infty(\C_+)$, 
so that the representation $\pi$ is weakly continuous.

If $g \in H^2(\C_+)$ is an outer function, then $H^\infty(\C_+)g$ is dense 
in $H^2(\C_+)$, so that $H^2(\C_+)$ is $\theta_h$-positive 
if and only if the functional 
$\eta_{h^\vee g^\sharp g}$ is positive. This in turn is equivalent to 
\[ h^\vee g^\sharp g  \in \psi_\mu + H^1(\C_+) \] 
for a finite positive Borel measure $\mu$ on $\R_+$ 
(cf.\ Example~\ref{ex:4.1}). In this 
case the representation $(\pi, \cH)$ is equivalent to the GNS 
representation associated to the positive functional $\eta_{h^\vee g^\sharp g}$ 
on $H^\infty(\C_+)$.
\end{rem}

\subsection{The unit group of $H^\infty$}
\mlabel{app:b.5} 

A non-negative function 
$w \: \T \to \R_+$ arises as 
$|f^*|$ for $f \in H^\infty(\bD)$ if and only if 
$\log w \in L^1(\T)$. This means that 
$w = e^h$ for $h \in L^1(\T,\R)$ bounded from above. 
If, in addition, $h \in L^{\infty}(\T,\R),$ 
then $e^{\pm h}$ are bounded. This leads to the following description 
of the unit group of $H^\infty(\bD)$. 
As all proper simply connected domains $\Omega \subeq \C$ 
are isomorphic, it also provides a description of the unit group 
$H^\infty(\Omega)^\times$ in general 
(see in particular Definition~\ref{def:4.3}). 

\begin{lem} \mlabel{lem:b.11} We have a surjective group homomorphism 
\[ \Out \:  (L^{\infty}(\T,\R),+) \times \T \to (H^\infty(\bD)^\times,\cdot), \quad 
\Out(w,\zeta) = \zeta e^{q_w}, \quad 
q_w(z) :=  \int_0^{2\pi} \frac{e^{it} + z}{e^{it} - z}
\, w(e^{it})\, dt.\] 
In particular, all invertible elements of $H^\infty(\bD)$ are outer. 
\end{lem}

\begin{prf} It is clear that $Q$ is a group homomorphism 
whose range consists of outer functions which are invertible 
in $H^\infty(\bD)$. 
If, conversely,  $f \in H^\infty(\bD)$ is invertible, 
then the subspace 
$f H^2(\bD)$ coincides with $H^2(\bD)$. As $f \in H^2(\bD)$, it 
follows that $f H^\infty(\bD)$ is dense in $H^2(\bD)$, and hence that 
$f$ is outer, i.e., of the form $\zeta e^{q_w}$ for some $w \in L^1(\T,\R)$. 
As $f$ and $f^{-1}$ are bounded, $w$ is bounded as well, hence 
contained in $L^{\infty}(\T,\R)$. Therefore $\Out$ is surjective. 
\end{prf}

Finer results that imply the preceding lemma can be found in 
\cite[Part~A, \S 4.2]{Ni02}. For results on the operator-valued 
case, see \cite[Part~A, \S 3.3]{Ni02}.

\subsection{The representation on $H^2(\Omega)$} 
\mlabel{subsec:b.6}

If $\phi \: \Omega_1 \to \Omega_2$ is a biholomorphic map between 
simply connected proper domains in $\C$, then the map 
\begin{equation}
  \label{eq:unih2}
\Gamma_\phi \: H^2(\Omega_2) \to H^2(\Omega_1), \quad 
\Gamma_\phi(f) := \sqrt{\phi'} \cdot f \circ \phi 
\end{equation}
is unitary up to a positive factor, depending on the normalization 
of the scalar product. Here 
$\sqrt{\phi'}$ denotes one of the two holomorphic square 
roots of $\phi' \:\Omega_1 \to \C^\times$. Its existence follows from the 
simple connectedness of $\Omega_1$. 
Actually, this is how one can define 
$H^2(\Omega)$ for domains with a complicated boundary in a natural way. 

Clearly, $\Gamma_\phi$ intertwines the multiplication action 
of $H^\infty(\Omega_2)$ on $H^2(\Omega_2)$ with the action of 
$H^\infty(\Omega_1)$ on $H^2(\Omega_1)$ in the sense that 
\begin{equation}
  \label{eq:hinfh2inter}
 \Gamma_\phi(fg) = (f \circ \phi) \cdot \Gamma_\phi(g) \quad 
\mbox{ for } \quad f \in H^\infty(\Omega_2), g \in H^2(\Omega_2).
\end{equation}

To see how the Szeg\"o kernels on $\Omega_1$ and $\Omega_2$ are related, 
we observe that, for $f \in H^2(\Omega_2)$ and $z \in \Omega_1$ the relation 
\[  \sqrt{\phi'(z)} \la Q^{\Omega_2}_{\phi(z)}, f \ra
=  \sqrt{\phi'(z)} f\big(\phi(z)) 
=  \Gamma_\phi(f)(z) 
=\la Q^{\Omega_1}_z, \Gamma_\phi(f) \ra 
=\la \Gamma_\phi^{-1} Q^{\Omega_1}_z, f \ra \] 
implies that 
\[ Q^{\Omega_1}_z 
= \oline{\sqrt{\phi'(z)}} \Gamma_\phi(Q^{\Omega_2}_{\phi(z)}),\] 
which leads to the following transformation formula for the kernels 
\[ Q^{\Omega_1}(z,w) = \sqrt{\phi'(z)}
Q^{\Omega_2}(\phi(z), \phi(w)) \oline{\sqrt{\phi'(w)}}.\] 

\begin{ex}
For $\Omega_1 =\bD$ and $\Omega_2 = \C_+$ we have 
\[  Q^{\Omega_1}(z,w) = \frac{1}{2\pi}  \frac{1}{1 - z \oline w}
 \quad \mbox{ and } \quad 
Q^{\Omega_2}(z,w) = \frac{1}{2\pi} \frac{i}{z - \oline w}.\] 
The Cayley transform 
$ \omega \: \bD \to \C_+,\omega(z) := i \frac{1 + z}{1-z}$ with 
$\omega'(z) =  \frac{2i}{(1-z)^2}$ 
satisfies 
\[ \sqrt{\omega'(z)}Q^{\C_+}(\omega(z), \omega(w)) 
\oline{\sqrt{\omega'(w}}
= \frac{\sqrt{2}}{(1-z)}
\frac{1}{2\pi}\frac{1}{\frac{1 + z}{1-z} + \frac{1 + \oline w}{1- \oline w}}
\frac{\sqrt{2}}{(1-\oline w)}
= \frac{1}{2\pi}\frac{1}{1- z \oline w} 
= Q^\bD(z,w).\] 
\end{ex}

\subsection{The Carleson measure} 
\mlabel{subsec:b.7}

The Carleson measure $\mu_H$ of a positive Hankel operator 
$H$ on $H^2(\Omega)$ with respect to the multiplication 
representation of the Banach $*$-algebra 
$(H^\infty(\Omega), \sharp)$ lives 
on the subset $\Omega^\sigma$ of $\sigma$-fixed points 
(cf.\ Proposition~\ref{prop:4.6}). 
The abstract correspondence is 
\begin{equation}
  \label{eq:hankmeas}
\la f, H g \ra_{H^2(\Omega)} 
= \int_{\Omega^\sigma} \oline{f(\lambda)} g(\lambda)\, d\mu(\lambda)
\quad \mbox{ for }\quad f,g \in H^2(\Omega).
\end{equation}
However, it is more convenient to parametrize this subset by real 
intervals. Formula \eqref{eq:hankmeas} shows in particular 
that Hankel operators on $H^2(\Omega)$ are superpositions of 
rank-one Hankel operators $H_\lambda$ corresponding to point 
measures $\delta_\lambda$, 
$\lambda \in \Omega^\sigma$. 
If $Q$ is the reproducing kernel of $H^2(\Omega)$, then the relation 
\[ \la f, H_\lambda g \ra_{H^2(\Omega)} 
=  \oline{f(\lambda)} g(\lambda) 
= \la f, Q_\lambda \ra \la Q_\lambda, g \ra \] 
shows that 
\[ H_\lambda  = |Q_\lambda \ra \la Q_\lambda| \] 
in Dirac's bra-ket notation. This means that 
\begin{equation}
  \label{eq:hlambda}
   H_\lambda(f)  = f(\lambda) Q_\lambda 
\quad \mbox{ with }\quad 
\| H_\lambda\| = \|Q_\lambda\|^2 = Q(\lambda,\lambda).
\end{equation}
Formula \eqref{eq:hankmeas} can now be written as a weak integral 
\[  H = \int_{\Omega^\sigma} H_\lambda\, d\mu(\lambda),\]
which exists pointwise in the space of sesquilinear forms 
on $H^2(\Omega)$.\begin{footnote}
{See \cite[\S 6.3.1]{Ni02} for the case of the disc and for 
more general Carleson measures.}  
\end{footnote}

The symbol kernel of the Hankel operator $H$ with respect to the Szeg\"o kernel is 
the kernel 
\[ Q_H(z,w) 
=  \la Q_z, H Q_w \ra 
=  \int_{\Omega^\sigma} \oline{Q_z(\lambda)} Q_w(\lambda)\, d\mu(\lambda) 
=  \int_{\Omega^\sigma} Q(z,\lambda) Q(\lambda,w)\, d\mu(\lambda)\] 
and 
\[ Q_{H_\lambda}(z,w) 
=   Q(z,\lambda) Q(\lambda,w).\] 

\begin{rem} The norm of $H$ can be determined in terms of the kernel 
$Q_H$ by 
\[ \|H\| = \inf \{ c > 0 \: c Q - Q_H\ \mbox{ positive def.} \}\]
(\cite{Ne99}). In some situations this number can be determined by 
restricting to finite subsets. 
\end{rem}

\begin{rem} If $H$ is a positive Hankel operator on $H^2(\Omega)$ 
and $\mu_H$ the corresponding Carleson measure on $\Omega^\sigma$, 
then we obtain for every $g \in H^2(\Omega)$ a finite positive 
measure $d\mu_g = |g(\lambda)|^2\, d\mu_H(\lambda)$ representing 
a positive weakly continuous functional on the 
Banach $*$-algebra $(H^\infty(\Omega), \sharp)$: 
\[ \phi_{\mu_g}(f) := \la g,  H fg \ra 
= \int_{\Omega^\sigma} f(\lambda) |g(\lambda)|^2\, d\mu_H(\lambda) 
= \int_{\Omega^\sigma} f(\lambda)\, d\, \mu_g(\lambda)\] 
(cf.\ Proposition~\ref{prop:4.6}).
\end{rem}

\section{Cauchy and Poisson kernels} 
\mlabel{app:k} 

For a proper simply connected domain $\Omega \subeq \C$, 
the Hardy space $H^2(\Omega)$ 
(cf.\ Subsection~\ref{subsec:b.6}) is a reproducing kernel 
Hilbert space, i.e., the point evaluations 
\[  \ev_z \: H^2(\Omega) \to \C, \quad f \mapsto f(z) \]
are continuous linear functionals, hence can be written as 
\[ f(z) = \la Q_z, f \ra \quad \mbox{ for some } \quad Q_z \in H^2(\Omega).\] 
The kernel 
\[ Q \: \Omega \times \Omega \to \C, \quad Q(z,w) := Q_w(z) = \la Q_z, Q_w \ra \] 
is called the {\it Szeg\"o kernel of $\Omega$}. 

If $\Omega$ has smooth boundary, so that we have an isometric 
boundary value map 
\[ H^2(\Omega) \to L^2(\partial \Omega), \quad 
f \mapsto f^*,\] 
then we obtain  the {\it Poisson kernel} of $\Omega$ by 
the {\it Hua formula} (cf.\ \cite[pp.~8,98]{Hu63}, \cite{Ko65}) 
\begin{equation}
  \label{eq:poiss-omega}
P \: \Omega \times \partial \Omega  \to \R, \quad 
 P(z,x) = \frac{|Q(z,x)|^2}{Q(z,z)} \quad \mbox{ for } \quad 
z \in \Omega, x \in \partial\Omega,
\end{equation}
i.e., 
\[ P(z,\cdot) = \frac{|Q_z^*|^2}{Q(z,z)} \in L^1(\partial \Omega).\]  

\begin{ex}
  \mlabel{app:k.1}
The Szeg\"o kernel of the disc is 
\[ Q(z,w) = \frac{1}{2\pi} \frac{1}{1 - z \oline w}.\] 
For $f \in H^2(\bD)$, we have 
\[  f(z) = \la Q_z, f \ra 
=  \int_0^{2\pi} \oline{Q_z}(e^{i\theta}) f^*(e^{i\theta})\, d\theta 
= \frac{1}{2\pi} \int_{\partial \bD} \frac{f^*(\zeta)}{1 - \oline\zeta z} \, 
\frac{d\zeta}{i\zeta} 
= \frac{1}{2\pi i} \int_{\partial \bD} \frac{f^*(\zeta)}{\zeta - z} \, d\zeta.\]
We thus obtain from \eqref{eq:poiss-omega} the Poisson kernel 
\[ P(re^{i\theta}, e^{it}) 
= \frac{1}{2\pi} \frac{1-r^2}{|1 - re^{i(\theta-t)}|^2} 
= \frac{1}{2\pi}\frac{1-r^2}{1 - 2r \cos(\theta-t) + r^2}
\quad \mbox{ for } \quad z = r e^{i\theta}\in \bD, t \in [0,2\pi].\] 
\end{ex}

\begin{ex}
\mlabel{app:k.2} 
The Szeg\"o kernel on the upper half-plane is  given by 
\begin{equation}
  \label{eq:cauchyker}
 Q(z,w) = \frac{1}{2\pi} \frac{i}{z - \oline w} 
\quad \mbox{ for } \quad z,w \in \C_+.
\end{equation}
This is an easy consequence of the Residue Theorem. We have 
\[ f(z) = \la Q_z, f \ra = \frac{1}{2\pi i} \int_\R \frac{f^*(x)}{x -z}\, dx 
\quad \mbox{ for } \quad 
f \in H^2(\C_+), z \in \C_+.\] 
For the Poisson kernel we obtain with 
Hua's formula \eqref{eq:poiss-omega}
\begin{equation}
  \label{eq:poissker}
 P(z,x) = P_z(x) = \frac{1}{\pi} \frac{\Im z}{|z-x|^2}.
\end{equation}
\end{ex}

\vspace{1cm}

\nin Institut f\"{u}r Theoretische Physik, University of Leipzig, Br\"{u}derstrasse 16, 04103 Leipzig, Germany;
adamo@axp.mat.uniroma2.it; mariastella.adamo@community.unipa.it\\

\nin Department of Mathematics, Friedrich-Alexander-University of Erlangen-N\"urnberg, Cauerstrasse 11, 91058 Erlangen, Germany; neeb@math.fau.de\\

\nin Department of Mathematics, Friedrich-Alexander-University of Erlangen-N\"urnberg, Cauerstrasse 11, 91058 Erlangen, Germany; jonas.schober@fau.de


\begin{thebibliography}{aaaaaaaa} 

\bibitem[An78]{An78} Ando, T., {\it On the predual of $H^\infty$}, 
Special issue dedicated to Wladyslaw Orlicz on the occassion of his seventy-fifth 
birthday, Comment. Math. Special Issue {\bf 1} (1978), 33--40 

\bibitem[BCR84]{BCR84} Berg, C., Christensen, 
J. P. R., and P. Ressel, ``Harmonic Analysis on Semigroups,'' 
Graduate Texts in Math., Springer-Verlag, Berlin,
Heidelberg, 1984 


\bibitem[Ca62]{Ca62} Carleson, L., {\it Interpolations by bounded 
analytic functions and the Corona Problem}, 
Annals of Math. {\bf 76:3} (1962), 547--559

\bibitem[Co87]{Co87} Cooper, J.B., ``Saks Spaces and Applications to Functional 
Analysis,'' 2nd revised ed., North Holland Mathematics Studies {\bf 139}; 
Notas de Matem\'atica (116), 1987

\bibitem[Dix64]{Dix64} Dixmier, J., ``Les $C^*$-alg\`ebres et leurs
repr\'esentations,'' Gauthier-Villars, Paris, 1964. 

\bibitem[Do74]{Do74} Donoghue, W.F.Jr., ``Monotone Matrix Functions 
and Analytic Continuations,'' Grundlehren der math. Wiss. {\bf 207}, 
Springer, New-York etc., 1974

\bibitem[GP15]{GP15} G\'erard, P., and A. Pushnitski, {\it An inverse problem 
for self-adjoint positive Hankel operators}, 
Int. Math. Res. Not. IMRN {\bf 13} (2015), 4505--4535

\bibitem[GJ81]{GJ81} Glimm, J., and A. Jaffe, ``Quantum Physics--A Functional Integral Point of View,''
Springer-Verlag, New York, 1981

\bibitem[Ja08]{Ja08} Jaffe, A., {\it Quantum theory and relativity},  in ``Group Representations,
Ergodic Theory, and Mathematical Physics: A Tribute to
George~W. Mackey,'' R. S. Doran, C. C. Moore, R. J. Zimmer, eds.,
Contemp. Math. {\bf 449}, Amer. Math. Soc., 2008


\bibitem[Gr55]{Gr55} Grothendieck, A., {\it Une caract\'erisation 
vectorielle-m\'etrique des espaces $L^1$}, Canad. J. Math. {\bf 7} (1955), 
552--561 

\bibitem[Hu63]{Hu63} Hua, L., ``Harmonic Analysis of Functions of Several 
Complex Variables in the Classical Domains,'' Transl. of Math. Monographs 
{\bf 6}, Amer. Math. Soc., Providence, Rhode Island, 1963 



\bibitem[JOl98]{JOl98} Jorgensen, P. E. T., and G. \'Olafsson, {\it Unitary representations of Lie groups with
reflection symmetry}, J. Funct. Anal. {\bf 158} (1998), 26--88 

\bibitem[JOl00]{JOl00} Jorgensen, P. E. T., and G. \'Olafsson, 
 {\it Unitary representations and Osterwalder-Schrader
duality}, in ``The Mathematical Legacy of Harish--Chandra,'' 
R. S. Doran and V. S. Varadarajan, eds., Proc. Symp. in Pure Math. {\bf 68},  
Amer. Math. Soc., 2000 

\bibitem[KS78]{KS78}  Kerzman, N., and E. M. Stein, {\it 
The Cauchy kernel, the Szeg\"o kernel, and the 
Riemann mapping function}, Math. Annalen {\bf 236} (1978), 85--93 

\bibitem[Ko65]{Ko65} Kor\'anyi, A., {\it The Poisson integral for generalized 
half-planes and bounded symmetric domains}, Ann. Math.  (2) {\bf 82} (1965), 
332--350

\bibitem[LP64]{LP64}
 Lax, P. D.,  and R. S.  Phillips,  \textit{Scattering theory}, Bull. Amer. Math. Soc. \textbf{70} (1964), 130--142 
 
\bibitem[LP67]{LP67}  Lax, P. D.,  and R. S.  Phillips, ``Scattering theory'', Pure and Applied Mathematics, \textbf{26} Academic Press, New York-London, 1967

\bibitem[LP81]{LP81}  Lax, P. D.,  and R. S.  Phillips,
 \textit{The translation representation theorem},
Integral Equations Operator Theory \textbf{4} (1981),  416--421 

\bibitem[Ne99]{Ne99} Neeb, K.-H., ``Holomorphy and Convexity in Lie Theory'', 
Expositions in Mathematics {\bf 28}, de Gruyter Verlag, 1999


\bibitem[N\'O14]{NO14} Neeb, K.-H., G. \'Olafsson,  {\it Reflection 
positivity and conformal symmetry}, J. Funct. Anal.   {\bf 266} (2014), 2174--2224

\bibitem[N\'O15]{NO15} Neeb, K.-H., G. \'Olafsson,   {\it Reflection 
positive one-parameter groups and dilations}, Complex 
Analysis and Operator Theory {\bf 9:3} (2015), 653--721 


\bibitem[N\'O18]{NO18} Neeb, K.-H., G. \'Olafsson, 
``Reflection Positivity. A Representation Theoretic Perspective,'' 
Springer Briefs in Mathematical Physics {\bf 32}, 2018 

\bibitem[Ni02]{Ni02} Nikolski, N., ``Operators, Functions, and 
Systems: An Easy Reading. Volume I: Hardy, Hankel, and Toeplitz,'' 
Math. Surveys and Monographs {\bf 92}, Amer. Math. Soc., 2002 

\bibitem[Ni19]{Ni19} Nikolski, N., ``Hardy Spaces,'' Cambridge Univ. Press, 
2019 

\bibitem[Pa88]{Pa88} Partington, J., 
``An Introduction to Hankel Operators'', 
London Math. Soc. Student Texts {\bf 13}, 1988

\bibitem[Pe98]{Pe98} Peller, V.V., {\it An excursion into the theory 
of Hankel operators}, in ``Holomorphic Spaces'', MSRI Publications 
{\bf 33}, Ed. S. Axler, J. Mc Carthy, D. Sorason, 1998

\bibitem[Ro66]{Ro66} Rosenblum, M., {\it Self-adjoint Toeplitz operators}, 
``Summer Institute of Spectral Theory and 
Statistical Mechanics 1965,'' Brookhaven National Laboratory, Upton, NY, 1966

\bibitem[RR94]{RR94} Rosenblum, M., and J. Rosnyak, ``Topics on Hardy Classes 
and Univalent Functions,'' Birkh\"auser, 1994 

\bibitem[Ru86]{Ru86} 
Rudin, W., ``Real and Complex Analysis,'' McGraw Hill, 1986


\bibitem[Sh64]{Sh64} Shapiro, H., {\it Reproducing kernels and 
Beurling's Theorem}, Transactions of the Amer. Math. Soc. {\bf 110:3} 
(1964), 448--458 


\bibitem[SzNBK10]{SzNBK10} Sz.-Nagy, B., C. Foias, H. Bercovici, 
and L. K\'erchy, ``Harmonic
Analysis of Operators on Hilbert space,'' 2nd edition, 
Springer, Universitext, 2010 

\bibitem[Wi66]{Wi66} Widom, H., {\it Hankel matrices}, 
Trans. Amer. Math. Soc. {\bf 121} (1966), 1--35 

\end{thebibliography}
\end{document}